\documentclass{article}
\usepackage[utf8]{inputenc}
\usepackage [english]{babel}
\usepackage {geometry}
\usepackage {fancyhdr}
\usepackage {amsmath,amsfonts,amssymb,amsthm,epsfig,epstopdf,titling,url,array}
\usepackage {mathrsfs}
\usepackage {graphicx}
\usepackage {hyperref}
\usepackage {enumerate}
\usepackage {wrapfig}
\usepackage {cleveref}
\usepackage {tikz-cd}

\tikzcdset{every label/.append style = {font = \small}}

\theoremstyle{plain}
\newtheorem{thm}{Theorem}[section]
\newtheorem{lem}[thm]{Lemma}
\newtheorem{prop}[thm]{Proposition}
\newtheorem{cor}[thm]{Corollary}

\theoremstyle{definition}
\newtheorem{defn}[thm]{Definition}

\theoremstyle{remark}

\newcommand{\Z}{\mathbb{Z}}

\newcommand{\C}{\mathbb{C}}

\newcommand{\LL}{\mathscr{L}}
\newcommand{\PP}{\mathbb{P}}
\newcommand{\OO}{\mathcal{O}}
\newcommand{\ZZ}{\mathcal{Z}}
\newcommand{\Tor}{\mathscr{T}or}
\newcommand{\Ext}{\mathscr{E}xt}
\newcommand{\Hom}{\mathscr{H}om}
\newcommand{\X}{\bar{X}}
\newcommand{\Y}{\bar{Y}}
\newcommand{\codim}{\text{codim}}
 
\begin{document}

\renewcommand{\abstractname}{\vspace{-\baselineskip}}

\title{Positivity in $T$-equivariant $K$-theory of flag varieties associated to Kac-Moody groups II}

\author{Seth Baldwin and Shrawan Kumar}

\maketitle

\begin{abstract}
We prove sign-alternation of the structure constants in the basis of structure sheaves of opposite Schubert varieties in the torus-equivariant Grothendieck group of coherent sheaves on the flag varieties $G/P$ associated to an arbitrary symmetrizable Kac-Moody group $G$, where $P$ is any parabolic subgroup.  This generalizes the work of Anderson-Griffeth-Miller from the finite case to the general Kac-Moody case, and affirmatively answers a conjecture of Lam-Schilling-Shimozono regarding the signs of the structure constants in the case of the affine Grassmannian.  
\end{abstract}

\section{Introduction}
Let $G$ be any symmetrizable Kac-Moody group over $\C$ completed along the negative roots and $G^{\text{min}}\subset G$ be the minimal Kac-Moody group as in \cite[\S7.4]{Kbook}.  Let $B$ be the standard (positive) Borel subgroup, $B^{-}$ the standard negative Borel subgroup, $H=B\cap B^{-}$ the standard maximal torus, and $W$ the Weyl group \cite[Chapter 6]{Kbook}.  Let $\X=G/B$ be the `thick' flag variety (introduced by Kashiwara) which contains the standard flag variety $X=G^{\text{min}}/B$.  Let $T$ be the adjoint torus, i.e., $T:=H/Z(G^{\text{min}})$, where $Z(G^{\text{min}})$ denotes the center of $G^{\text{min}}$ and let $R(T)$ denote the representation ring of $T$.  For any $w\in W$, we have the Schubert cell $C_w:=BwB/B\subset X$, the Schubert variety $X_w:=\overline{C_w}\subset X$, the opposite Schubert cell $C^w:=B^{-}wB/B\subset \X$, and the opposite Schubert variety $X^w :=\overline{C^w}\subset\X$.  When $G$ is a (finite dimensional) semisimple group, it is referred to as the \emph{finite case}.  

Let $K^0_T(\X)$ denote the Grothendieck group of $T$-equivariant coherent sheaves on $\X$.  Then, $\{[\OO_{X^w}]\}_{w\in W}$ forms an $R(T)$-`basis' of $K^0_T(\X)$ (where infinite sums are allowed), i.e., $K^0_T(\X)=\prod_{w\in W}R(T)[\OO_{X^w}]$.  We express the product in $K^0_T(\X)$ by:
$$[\OO_{X^u}]\cdot[\OO_{X^v}]=\sum_{w\in W}d_{u,v}^w[\OO_{X^w}],\,\,\,\text{for unique}\,\, d_{u,v}^w\in R(T).$$

The following result is our main theorem (cf. Theorem \ref{positivity}).  This was conjectured first by Griffeth-Ram \cite{GR} in the finite case (2004), proven in the finite case by Anderson-Griffeth-Miller \cite{AGM} (2011), and then conjectured in the general Kac-Moody case by Kumar \cite{K} (2012).   

\begin{thm}
\label{first}
For any $u,v,w\in W$, $$(-1)^{\ell(w)+\ell(u)+\ell(v)}d^w_{u,v}\in\Z_{\geq 0}[(e^{-\alpha_1}-1),\dots,(e^{-\alpha_r}-1)]$$ where $\{\alpha_1\ldots,\alpha_r\}$ are the simple roots.  
\end{thm}

Let $P$ be any standard parabolic subgroup of $G$ of finite type (cf. \cite[Definition 6.1.18]{Kbook}).  (Recall that a parabolic subgroup is said to be of {\it finite type} if its Levi subgroup is finite dimensional.) We may express, in $K^0_T(G/P)$, $$[\OO_{X^u_P}]\cdot[\OO_{X^v_P}]=\sum_{w\in W^P}d_{u,v}^w(P)[\OO_{X^w_P}],\,\,\,\text{for unique} \,\, d_{u,v}^w(P)\in R(T),$$ where $W^P$ is the set of minimal length representatives of $W/W_P$, $W_P$ is the Weyl group of $P$, and $X^w_P:=\overline{B^-wP/P}\subset G/P$ is the opposite Schubert variety.

Let $\pi:G/B\to G/P$ be the standard ($T$-equivariant) projection.  Then, $\pi$ is a locally trivial fibration (with fiber the smooth projective variety $P/B$) and hence flat (cf. \cite[Chapter 7]{Kbook}).  Thus, we have $$\pi^*[\OO_{X^w_P}]=[\OO_{\pi^{-1}(X^w_P)}]=[\OO_{X^w}].$$  Since $\pi^*:K^0_T(G/P)\to K^0_T(G/B)$ is a ring homomorphism, we have $d_{u,v}^w=d_{u,v}^w(P)$ for any $u,v,w\in W^P$ and hence Theorem \ref{first} immediately generalizes from the case of $G/B$ to the case of $G/P$, and we obtain: 

\begin{thm}
\label{second}
For any standard parabolic subgroup $P$ of $G$ of finite type, and any $u,v,w\in W^P$, $$(-1)^{\ell(w)+\ell(u)+\ell(v)}d^w_{u,v}(P)\in\Z_{\geq 0}[(e^{-\alpha_1}-1),\dots,(e^{-\alpha_r}-1)].$$
\end{thm}

Theorems \ref{first} and \ref{second} also apply to ordinary (non-equivariant) $K$-theory.  Let $K^0(\X)$ denote the Grothendieck group of coherent sheaves on $\X$.  Then, we have $K^0(\X)=\prod_{w\in W}\Z [\OO_{X^w}]$.  Further, the map $$\Z\otimes_{R(T)}K^0_T(\X)\to K^0(\X),\,\,\,1\otimes [\OO_{X^w}]\mapsto [\OO_{X^w}]$$ is an isomorphism, where we view $\Z$ as an $R(T)$-module via evaluation at $1$. 
Similar results apply to $G/P$.  

Write, in $K^0(G/P)$, for $u,v \in W^P$, $$[\OO_{X^u_P}]\cdot[\OO_{X^v_P}]=\sum_{w\in W^P}a_{u,v}^w(P)[\OO_{X^w_P}],\,\,\,\text{for unique}
\,\, a_{u,v}^w(P)\in \Z.$$  Then, by the above, along with Theorem \ref{second}  we have:

\begin{thm}
\label{ordinary}
For any standard parabolic subgroup $P$ of $G$ of finite type, and any $u,v,w\in W^P$, $$(-1)^{\ell(w)+\ell(u)+\ell(v)}a_{u,v}^w(P)\in \Z_{\geq 0}.$$
\end{thm}

The following conjecture of Lam-Schilling-Shimozono \cite[Conjectures 7.20 (ii) and 7.21 (iii)]{LSS} is a special case of Theorem \ref{ordinary}: 

Let $G=\widehat{SL_N}$ be the affine Kac-Moody group associated to $SL_N$, and let $P$ be its standard maximal parabolic subgroup.  Let $\bar{\mathcal{X}}=G/P$ be the corresponding infinite Grassmannian.  Then, $K^0(\bar{\mathcal{X}})$ has the structure sheaf `basis' $\{[\OO_{X^u}]\}_{u\in W^P}$ over $\Z$, where $W$ is the (affine) Weyl group of $G$ and $W^P$ is the set of minimal coset representatives in $W/W_o$ ($W_o$ being the finite Weyl group of $SL_N$).  Write, for any $u,v\in W^P$,  $$[\OO_{X^u}]\cdot[\OO_{X^v}]=\sum_{w\in W^P}b_{u,v}^w[\OO_{X^w}],\,\,\,\text{for unique integers}\,\, b^w_{u,v}.$$  Then, the following is a special case of Theorem \ref{ordinary}: 

\begin{cor}[Conjectured by Lam-Schilling-Shimozono]
$$(-1)^{\ell(u)+\ell(v)+\ell(w)}b^w_{u,v}\in\Z_{\geq 0 }.$$ 
\end{cor}

The proof of Theorem \ref{first} follows closely the work of Anderson-Griffeth-Miller \cite{AGM} and Kumar \cite{K}.  {\it Though several technical details had to be carefully addressed, e.g., Proposition \ref{ratlsing} proving the rational singularities of $\ZZ$.} Letting $K^T_0(X)$ denote the Grothendieck group of finitely supported $T$-equivariant coherent sheaves on $X$, there is a natural pairing $$\langle\text{ },\text{ }\rangle:K^0_T(\X)\otimes K^T_0(X)\to R(T),$$ coming from the $T$-equivariant Euler-Poincar\'{e} characteristic (cf. (\ref{pairing})).  Under this pairing, the bases $\{[\OO_{X^u}]\}$ of $K^0_T(\X)$ and $\{[\xi_w]\}$ of $K^T_0(X)$ are dual (cf. Proposition \ref{dual}), where $\xi_w:=\OO_{X_w}(-\partial X_w)$ and $\partial X_w:=X_w\setminus C_w$.  (By  Lemma \ref{xibasis}, $\xi_w$ is indeed a basis of  $K^T_0(X)$.)

Then, realizing the product structure constants in $K^0_T(\X)$ as the coproduct structure constants of the dual basis in $K^T_0(X)$ (cf. Lemma \ref{coprodlem}) allows the use of the above pairing and duality to express the structure constants in terms of certain cohomology groups.  Following \cite{AGM}, we introduce the `mixing space' $X_\PP$, which is a bundle over a product of projective spaces $\PP$ with fiber $X$.  This allows for the reduction from $T$-equivariant $K$-theory to non-equivariant $K$-theory.  Then, we introduce the `mixing group' $\Gamma$ (cf. Definition \ref{mixgrp}) whose action is sufficient to ensure a certain transverality needed to prove our main result.  

Theorem \ref{main} is our main technical result.  Part a) allows the structure constants to be realized as the Euler characteristic of a certain sheaf, while part b) shows that this sheaf has cohomology which is zero in all but a single term of known degree, immediately resulting in a determination of the sign of the Euler characteristic and hence the structure constants.  

Part a) of Theorem \ref{main} is proved in Section $6$.  The proof relies on some local $\Tor$ vanishing results which were proven in \cite{K}, as well as a reduction to finite dimensional schemes, where a transversality result due to Sierra is crucially used (cf. Theorem \ref{trans}).  

The proof of the more difficult part b) of Theorem \ref{main} is the content of Sections $7$ through $9$.  In Section $7$ we introduce the crucial scheme $\ZZ$ which comes with a projection to the mixing group $\Gamma$.  We also introduce a desingularization $\widetilde{\ZZ}$ of $\ZZ$ and prove that $\ZZ$
is irreducible, normal and  has rational singularities (cf. Propositions \ref{Z} and  \ref{ratlsing}). We further  introduce the divisor $\partial \ZZ \subset \ZZ$ which is shown to be Cohen-Macaulay (cf. Proposition \ref{dZCM}).  It is on the fibers of the projection $\pi:\ZZ\to \Gamma$ that the computation of the cohomology of the previously mentioned sheaf occurs.  

In Section $8$ the rational singularities of $\ZZ$ are used to apply a relative version of the Kawamata-Viehweg vanishing theorem (cf. Theorem \ref{KV}) to show that $R^i\pi_*\omega_{\ZZ}(\partial\ZZ)=0$ for all $i>0$ (cf. Corollary \ref{corR^i}).  Finally, in Section $9$, this vanishing of the higher direct images along with the semicontinuity theorem is used to prove vanishing of the relevant cohomology along the fibers of $\pi:\ZZ\to\Gamma$ 
(cf. Theorem \ref{thm8.2}) and thus conclude the proof of  part b) of Theorem \ref{main}.  \\

\textbf{Acknowledgements}.  The authors would like to thank D. Anderson for providing some clarifications and additional details regarding their paper \cite{AGM} and N. Mohan Kumar for several helpful correspondences.  This work was partially supported by NSF grant number DMS-1501094.

\section{Notation}
We work over the field $\C$ of complex numbers.  By a variety we mean an algebraic variety over $\C$ which is reduced, but not necessarily irreducible.  

Let $G$ be any symmetrizable Kac-Moody group over $\C$ completed along the negative roots (as opposed to completed along the positive roots as in \cite[Chapter 6]{Kbook}) and $G^{\text{min}}\subset G$ be the minimal Kac-Moody group as in \cite[\S7.4]{Kbook}.  Let $B$ be the standard (positive) Borel subgroup, $B^{-}$ the standard negative Borel subgroup, $H=B\cap B^{-}$ the standard maximal torus, and $W$ the Weyl group \cite[Chapter 6]{Kbook}.  Let $$\X=G/B$$ be the `thick' flag variety which contains the standard flag variety $X=G^{\text{min}}/B$.  

If $G$ is not of finite type, $\X$ is an infinite dimensional non quasi-compact scheme \cite[\S4]{Ka} and $X$ is an ind-projective variety \cite[\S7.1]{Kbook}.  The group $G^{\text{min}}$ and, in particular, the maximal torus $H$ acts on $\X$ and $X$.  

Let $T$ be the adjoint torus, i.e., $T:=H/Z(G^{\text{min}})$, where $Z(G^{\text{min}})$ denotes the center of $G^{\text{min}}$. (Recall that by \cite[Lemma 6.2.9(c)]{Kbook}, $Z(G^{\text{min}})=\{h\in H:e^{\alpha_i}(h)=1\text{ for all the simple roots }\alpha_i\}$.)  Then, the action of $H$ on $\X$ (and on $X$) descends to an action of $T$.  

For any $w\in W$ we have the Schubert cell $$C_w:=BwB/B\subset X ,$$ the Schubert variety $$X_w:=\overline{C_w}=\bigsqcup_{w'\leq w} C_{w'}\subset X ,$$ the opposite Schubert cell $$C^w:=B^{-}wB/B\subset \X ,$$ and the opposite Schubert variety $$X^w :=\overline{C^w}=\bigsqcup_{w'\geq w} C^{w'}\subset\X ,$$ all endowed with the reduced subscheme structures.  Then, $X_w$ is a (finite dimensional) irreducible projective subvariety of $X$ and $X^w$ is a finite codimensional irreducible subscheme of $\X$ \cite[\S7.1]{Kbook} and \cite[\S4]{Ka}.  We denote by $Z_w$ the Bott-Samelson-Demazure-Hansen (BSDH) variety as in \cite[\S7.1.3]{Kbook}, which is a $B$-equivariant desingularization of $X_w$ \cite[Proposition 7.1.15]{Kbook}.  Further, $X_w$ is normal, Cohen-Macaulay (CM for short) and has rational singularities \cite[Theorem 8.2.2]{Kbook}.  

We also define the boundary of the Schubert variety by $$\partial X_w:=X_w\setminus C_w$$ with the reduced subscheme structure.  Then, $\partial X_w$ is pure of codimension $1$ in $X_w$ and is CM (See the proof of Proposition \ref{dZCM}).  

For any $u\leq w$,  we have the \emph{Richardson variety} $$X^u_w:=X^u\cap X_w\subset X$$ endowed with the reduced subvariety structure.  By \cite[Proposition 6.6]{K}, $X^u_w$ is irreducible, normal and CM.  We denote by $Z^u_w$ the $T$-equivariant desingularization of $X^u_w$ as in \cite[Theorem 6.8]{K}.  By \cite[Theorem 3.1]{KSch}, $X^u_w$ has rational singularities (in fact it has Kawamata log terminal singularities).  

We denote the representation ring of $T$ by $R(T)$.  Let $\{\alpha_1,\ldots,\alpha_r\}\subset\mathfrak{h}^*$ be the set of simple roots, $\{\alpha^\vee_1,\ldots,\alpha^\vee_r\}\subset\mathfrak{h}$ the set of simple coroots, and $\{s_1,\ldots,s_r\}$ the corresponding set of simple reflections, where $\mathfrak{h}:=\text{Lie}(H)$.  Let $\rho\in\mathfrak{h}^*$ be any integral weight satisfying
$$\rho(\alpha^\vee_i)=1, \text{ for all } 1\leq i\leq r.$$
When $G$ is a finite dimensional semisimple group, $\rho$ is unique, but for a general Kac-Moody group $G$, it may not be unique.  

For any integral weight $\lambda$ let $\C_\lambda$ denote the one-dimensional representation of $H$ on $\C$ given by $h\cdot v=\lambda(h)v$ for $h\in H, v\in\C$.  By extending this action to $B$ we may define, for any integral weight $\lambda$, the $G$-equivariant line bundle $\mathcal{L}(\lambda)$ on $\X$ by $$\mathcal{L}(\lambda):=G\times^B \C_{-\lambda},$$ where for any representation $V$ of $B$, $G\times^B V:=(G\times V)/B$ where $B$ acts on $G\times V$ via $(g,v)\cdot b=(g b, b^{-1}v)$ for $g\in G, v\in V, b\in B$.  Then, $G\times^B V$ is the total space of a $G$-equivariant vector bundle over $X$, with projection given by $(g,v)B\mapsto gB$.  We also define the bundle $$e^{\lambda}:=\X\times \C_\lambda ,$$ which while trivial when viewed as a non-equivariant line bundle, is equivariantly non-trivial with the diagonal action of $H$.  

\section{The Grothendieck group}

For a quasi-compact scheme $Y$, an $\OO_Y$-module is called \emph{coherent} if it is finitely presented as an $\OO_Y$-module and any $\OO_Y$-submodule of finite type admits a finite presentation.  

A subset $S\subset W$ is called an \emph{ideal} if for all $x\in S$ and $y\leq x$ we have $y\in S$.  We say that an $\OO_{\X}$-module $\mathcal{S}$ is \emph{coherent} if $\mathcal{S}|_{V^S}$ is a coherent $\OO_{V^S}$-module for every finite ideal $S\subset W$, where $V^S$ is the quasi-compact open subset of $\bar{X}$ defined by $$V^S:=\bigcup_{w\in S}wU^{-}B/B,\,\,\,\text{where }U^{-}:=[B^{-1},B^{-1}].$$

Let $K^0_T(\X)$ denote the Grothendieck group of $T$-equivariant coherent $\OO_{\X}$-modules.  Since the coherence condition on $\mathcal{S}$ is imposed only for $\mathcal{S}|_{V^S}$ for finite ideals $S\subset W$, $K^0_T(\X)$ can be thought of as the inverse limit of $K^0_T(V^S)$ as $S$ varies over all finite ideals of $W$ (cf. \cite[\S2]{KS}).  

We define $$K^T_0(X):=\text{Limit}_{n\to\infty} K^T_0(X_n),$$ where $\{X_n\}_{n\geq 1}$ is the filtration of $X$ giving the ind-projective variety structure and $K^T_0(X_n)$ is the Grothendieck group of $T$-equivariant coherent sheaves on $X_n$.  

For any $u\in W$, $\OO_{X^u}$ is a coherent $\OO_{\X}$-module by \cite[\S2]{KS}.  From \cite[comment after Remark 2.4]{KS} we have:

\begin{lem}
\label{basis}  
$\{[\OO_{X^u}]\}$ forms a basis of $K^0_T(\X)$ as an $R(T)$-module (where we allow arbitrary infinite sums).  
\end{lem}

\begin{lem}
\label{basis2}  
$\{[\OO_{X_w}]\}$ forms a basis of $K^T_0(X)$ as an $R(T)$-module.  
\end{lem}

\begin{proof}
This follows from \cite[\S5.2.14 and Theorem 5.4.17]{CG}.
\end{proof}

The following lemma is due to Kashiwara-Shimozono \cite[Lemma 8.1]{KS}.  

\begin{lem}
\label{freeres}
Any $T$-equivariant coherent sheaf $\mathcal{S}$ on $V^u$ admits a free resolution in $\text{Coh}_T(\OO_{V^u})$: $$0\to S_n\otimes \OO_{V^u}\to\dots\to S_1\otimes\OO_{V^u}\to S_0\otimes\OO_{V^u}\to\mathcal{S}\to 0 ,$$ where $S_k$ are finite dimensional $T$-modules, $V^u:=uU^{-}B/B\subset \X$, and $\text{Coh}_T(\OO_{V^u})$ denotes the abelian category of $T$-equivariant coherent $\OO_{V^u}$-modules.
\end{lem}

We define a pairing 
\begin{equation}
\label{pairing}
\langle\text{ },\text{ }\rangle:K^0_T(\X)\otimes K^T_0(X)\to R(T),
\end{equation}
$$\langle[\mathcal{S}],[\mathcal{F}]\rangle=\sum_i(-1)^i\chi_T(X_n,\Tor^{\OO_{\X}}_i(\mathcal{S},\mathcal{F})),$$ where $\mathcal{S}$ is a $T$-equivariant coherent sheaf on $\X$ and $\mathcal{F}$ is a $T$-equivariant coherent sheaf on $X$ supported on $X_n$ for some $n$, where $\chi_T$ represents the $T$-equivariant Euler-Poincar\'{e} characteristic.  By \cite[Lemma 3.5]{K} this is well defined.  

\begin{defn}
\label{xidef}
We define $\xi_w$ to be the ideal sheaf of $\partial X_w$ in $X_w$,  where $\partial X_w$ is given the reduced subscheme structure: $$\xi_w:=\OO_{X_w}(-\partial X_w).$$
\end{defn}

\begin{lem}
\label{xibasis}
$\{[\xi_w]\}$ forms a basis of $K^T_0(X)$ as an $R(T)$-module.
\end{lem}

\begin{proof}
This follows since $[\xi_w]=[\OO_{X_w}]+\sum_{w'< w}r_{w'}[\OO_{X_{w'}}]$,  for some $r_{w'}\in R(T)$ and the fact that $[\OO_{X_w}]$ is a basis of 
$K^T_0(X)$ (cf. Lemma \ref{basis2}).
\end{proof}

\begin{prop}
\label{xiomega}
$\omega_{X_w}=e^{-\rho}\LL(-\rho)\xi_w$, where $\omega_{X_w}$ is the dualizing sheaf of $X_w$.
\end{prop}

\begin{proof}
This follows from \cite[Proposition 2.2]{GK} since the same proof works for general Kac-Moody groups.
\end{proof}

From \cite[Lemma 5.5]{K} we have the following result:

\begin{lem}
\label{tor1}
For any $u,w\in W$, we have $$\Tor_i^{\OO_{\X}}(\OO_{X^u},\OO_{X_w})=0,\text{ }\forall i>0 .$$
\end{lem}

We now prove:

\begin{lem}
\label{tor3}
For any $u\in W$ and any finite union of Schubert varieties $Y=\bigcup_{i=1}^{\ell} X_{w_i}$ we have $$\Tor^{\OO_{\X}}_i(\OO_{X^u},\OO_Y)=0,\text{ }\forall i>0 .$$ In particular, for any $u,w\in W$ we have $$\Tor^{\OO_{\X}}_i(\OO_{X^u},\OO_{\partial X_w})=0,\text{ }\forall i>0 .$$
\end{lem}

\begin{proof}
We proceed by double induction on the dimension of $Y$ (i.e. the largest dimension of the irreducible components of $Y$) and the number of irreducible components of $Y$.  If $\dim Y=0$,  then $Y=X_e$ and so the result follows from Lemma \ref{tor1}.  Now, suppose that $\dim Y=d$ and $Y$ has $k$ irreducible components.  If $k=1$ then the result follows from Lemma \ref{tor1}, so we may assume that $k\geq 2$.  Let $Y_1=X_w$ be an irreducible component of $Y$ and let $Y_2$ be the union of all the other irreducible components.  

By \cite[Proposition 5.3 and its proof]{KSch} $X$ is Frobenius split compatibly splitting its Schubert varities $X_u$ and also Richardson varieties $X^v_u$ in $\X$.  $$\text{Thus, any finite intersection } X_{u_1}\cap\ldots\cap X_{u_n}\cap X^{v_1}\cap\ldots\cap X^{v_m} \text{ is reduced for any }n\geq 1.\,\,\,(*)$$  (In the proof here we have only used $X_{u_1}\cap\ldots\cap X_{u_n}$ to be reduced, but the more general assertion here will be used later in the paper.)

The short exact sequence of sheaves $$0\to \OO_Y\to\OO_{Y_1}\oplus\OO_{Y_2}\to \OO_{Y_1\cap Y_2}\to 0$$ yields the long exact sequence $$\ldots\to\Tor^{\OO_{\X}}_{i+1}(\OO_{X^u},\OO_{Y_1\cap Y_2})\to\Tor^{\OO_{\X}}_{i}(\OO_{X^u},\OO_Y)\to\Tor^{\OO_{\X}}_{i}(\OO_{X^u},\OO_{Y_1}\oplus\OO_{Y_2})\to\ldots$$

Now, since $Y_2$ has less than $k$ irreducible components, induction on the number of irreducible components gives $$\Tor^{\OO_{\X}}_{i}(\OO_{X^u},\OO_{Y_1}\oplus\OO_{Y_2})=0, \text{ }\forall i>0 .$$
Since $Y_1\cap Y_2$ is reduced and of dimension less than $d$, induction on the dimension gives $$\Tor^{\OO_{\X}}_{i+1}(\OO_{X^u},\OO_{Y_1\cap Y_2})=0,\text{ }\forall i>0 .$$
Together, these imply the lemma.
\end{proof}

\begin{lem}
\label{tor2}
For any $u,w\in W$, we have $$\Tor_i^{\OO_{\X}}(\OO_{X^u},\xi_w)=0,\text{ }\forall i>0 .$$
\end{lem}

\begin{proof}
Applying Lemmas \ref{tor1} and \ref{tor3}, the desired result follows from the long exact sequence for $\Tor$.
\end{proof}

\begin{prop}
\label{dual}
For any $u,w\in W$ we have $$\langle[\OO_{X^u}],[\xi_w]\rangle=\delta_{u,w} .$$
\end{prop}

\begin{proof}
By definition $$\langle[\OO_{X^u}],[\xi_w]\rangle=\sum_i(-1)^i\chi_T(X_n,\Tor^{\OO_{\X}}_i(\OO_{X^u},\xi_w))$$ where $n$ is taken such that $n\geq \ell(w)$.  Thus, by Lemma \ref{tor2}, we have $$\langle[\OO_{X^u}],[\xi_w]\rangle=\chi_T(X_n,\OO_{X^u}\otimes_{\OO_{\X}}\xi_w).$$

By Lemma \ref{tor3} we have the sheaf exact sequence $$0\to \OO_{X^u}\otimes_{\OO_{\X}} \xi_w \to \OO_{X^u}\otimes_{\OO_{\X}} \OO_{X_w} \to \OO_{X^u}\otimes_{\OO_{\X}} \OO_{\partial X_w} \to 0.$$ Observe that by $(*)$
$$\OO_{X^u}\otimes_{\OO_{\X}} \OO_{X_w} = \OO_{X^u \cap X_w},$$
and similarly for $\OO_{X^u}\otimes_{\OO_{\X}} \OO_{\partial X_w}.$  Thus,  $$\chi_T(X_n,\OO_{X^u}\otimes_{\OO_{\X}} \xi_w)=\chi_T(X_n,\OO_{X^u\cap X_w})-\chi_T(X_n,\OO_{X^u\cap \partial X_w}).$$

When non-empty,  $X^u\cap X_w$ is irreducible by \cite[Proposition 6.6]{K}, and thus $X^u\cap \partial X_w=\bigcup_{u\leq w'< w}X^u\cap X_{w'}$ is a connected projective variety when non-empty, since $u\in X^u\cap X_{w'}$ for all $u\leq w'< w$.  

If $u\nleq w$ we have $X^u\cap X_w=\emptyset$, so assume $u\leq w$.  In this case $X^u\cap X_w\neq \emptyset$.  Now, if $u=w$ then $X^u\cap \partial X_w=\emptyset$, while if $u<w$ then $X^u\cap\partial X_w\neq \emptyset$.  By \cite[Corollary 3.2]{KSch}, $$H^i(X_n,\OO_{X^u\cap X_w})=0, \text{ }\forall i>0 .$$
Using an inductive argument similar to Lemma \ref{tor3} we obtain $$H^i(X_n,\OO_{X^u\cap Y})=0, \text{ }\forall i>0 ,$$ where $Y$ is any finite union of Schubert varieties.  Taking $Y=X^u\cap \partial X_w$ and combining the above implies the proposition (here we use that, when non-empty, $X^u\cap Y$ is connected).
\end{proof}

\section{The mixing space and mixing group}

In this section we realize the product structure constants of $\{[\OO_{X^u}]\}$ in $K^0_T(\X)$ as the coproduct structure constants of the dual basis $\{[\xi_u]\}$ in $K^T_0(X)$ (Lemma \ref{coprodlem}).  We then introduce the mixing space $X_\PP$, which is a bundle over a product of projective spaces $\PP$, with fiber $X$.  This allows for the reduction from $T$-equivariant $K$-theory to non-equivariant $K$-theory.  Using the pairing and duality introduced in the previous section, we realize the structure constants in terms of certain cohomology groups (cf. Lemma \ref{lemcj}).  Finally, we introduce the mixing group $\Gamma$ whose action is sufficient to allow for a transversality result necessary to prove part a) of our main technical result (Theorem \ref{main}).  

\begin{defn} (Structure constants $d^w_{u,v}$) By Lemma \ref{basis}, in $K^0_T(\X)$ we have: 

\begin{equation}
\label{eqd}
[\OO_{X^u}]\cdot[\OO_{X^v}]=\sum_{w\in W}d_{u,v}^w[\OO_{X^w}], \,\,\,\text{for unique}\,\, d_{u,v}^w\in R(T), 
\end{equation}
where infinitely many of $d_{u,v}^w$ may be nonzero. 
\end{defn}

\begin{lem}
\label{coprodlem}
Write, in $K_0^T(X \times X)$ under the diagonal action of $T$  on $X \times X$, 
\begin{equation}
\label{eqe}
\Delta_*[\xi_w]=\sum_{u,v} e^w_{u,v}[\xi_u\boxtimes\xi_v],\,\,\,\text{for} \,\, e^w_{u,v}\in R(T),
\end{equation}
 where $\Delta:X\to X\times X$ is the diagonal map.  Then, $e^w_{u,v}=d^w_{u,v}.$
\end{lem}

\begin{proof}
Let $\bar{\Delta}:\X\to\X\times\X$ be the diagonal map, and note that $\bar{\Delta}^*[\OO_{X^u}\boxtimes\OO_{X^v}]=[\OO_{X^u}]\cdot[\OO_{X^v}]$.  Hence, we have:

\begin{align*}
d^w_{u,v} &= \langle\sum_{w'\in W}d_{u,v}^{w'}[\OO_{X^{w'}}],[\xi_w]\rangle, \,\,\,\text{by Proposition \ref{dual}}\\\\
&=\langle[\OO_{X^u}]\cdot[\OO_{X^v}],[\xi_w]\rangle, \,\,\,\text{by (\ref{eqd})}\\\\
&=\langle\bar{\Delta}^*[\OO_{X^u}\boxtimes\OO_{X^v}],[\xi_w]\rangle \\\\
&=\langle[\OO_{X^u}\boxtimes\OO_{X^v}],\Delta_*[\xi_w]\rangle \\\\
&=\langle[\OO_{X^u}\boxtimes\OO_{X^v}],\sum_{u',v'} e^w_{u',v'}[\xi_{u'}\boxtimes\xi_{v'}]\rangle, \,\,\,\text{by (\ref{eqe})}\\
&=e^w_{u,v}, \,\,\,\text{by Proposition \ref{dual}.}
\end{align*}

\end{proof}

We now prepare to define the mixing space.  Let $\PP:=(\PP^N)^r$ where $r=\dim T$ and $N$ is some large fixed integer.  Let $[N]=\{0,1,\dots,N\}$ and let $j=(j_1,j_2,\dots,j_r)\in[N]^r$.  We define $$\PP^j:=\PP^{N-j_1}\times\dots\times\PP^{N-j_r}$$ and similarly define $$\PP_j:=\PP^{j_1}\times\dots\times\PP^{j_r}.$$  We also define the boundary of $\PP_j$ by $$\partial\PP_j:=\left(\PP^{j_1-1}\times\PP^{j_2}\times \dots\times\PP^{j_r}\right)\cup\dots\cup\left(\PP^{j_1}\times\dots\times \PP^{j_{r-1}}\times \PP^{j_r-1}\right),$$ where we interpret $\PP^{-1}:=\emptyset$ to be the empty set.  \emph{Throughout this paper we will identify} $T\simeq(\C^*)^r$ \emph{via} $t\mapsto(e^{-\alpha_1}(t),\dots,e^{-\alpha_r}(t))$.  

\begin{defn} (Mixing space $X_\PP$)
Let $E(T)_\PP:=(\C^{N+1}\setminus\{0\})^r$ be the total space of the standard principal $T$-bundle $E(T)_\PP\to \PP$.  We can view $E(T)_\PP\to \PP$ as a finite dimensional approximation of the classifying bundle for $T$.  Define $X_\PP:=E(T)_\PP\times^T X$ and let $p:X_\PP\to \PP$ be the Zariski-locally trivial fibration with fiber $X=G^{\text{min}}/B$.
\end{defn}

For any $T$-scheme $V$ we define $V_\PP:=E(T)_\PP\times^T V$.

\vskip1ex

{\it For the rest of this paper we use the notation $Y:=X\times X$ and $\Y:=\X\times \X$.  Further, we let $Y_\PP:=E(T)_\PP\times^T Y\simeq X_\PP\times_\PP X_\PP$, where $T$ acts diagonally on $Y$. }
\vskip1ex

Note that for any $u,v\in W$, $(X_u)_\PP$ and $(X_u\times X_v)_\PP$ are CM, as they are fiber bundles over $\PP$, and hence locally a product of CM schemes.  Thus, they have dualizing sheaves.  

\begin{defn}
We define the sheaf $(\xi_u)_\PP$ on $X_\PP$ by $$(\xi_u)_\PP:=(e^\rho\LL(\rho))_\PP\otimes \omega_{(X_u)_\PP} ,$$ where $(e^\rho\LL(\rho))_\PP$ is defined by $$E(T)_\PP\times^Te^\rho\LL(\rho)\to X_\PP.$$
\end{defn}

\begin{prop}
$K_0(X_\PP):=\text{Limit}_{n\to\infty}K_0((X_n)_\PP)$ is a free module over the ring $K_0(\PP)=K^0(\PP)$ with basis $\{[(\xi_w)_\PP]\}$.

Thus, $K_0(X_\PP)$ has $\Z$-basis $$\left\{p^*([\OO_{\PP^j}])\cdot[(\xi_w)_\PP]\right\}_{j\in[N]^r, w\in W},$$
where, as above,  $p:X_\PP\to\PP$ is the projection.
\end{prop}

\begin{proof}
This follows from \cite[\S5.2.14 and Theorem 5.4.17]{CG} as well as the fact that the transition matrix between the basis $\{[\OO_{X_w}]\}$ and $\{[\xi_w]\}$ is upper triangular with invertible diagonal entries.
\end{proof}

\begin{defn}
We define the sheaf $\widetilde{\xi_u\boxtimes\xi_v}$ on $Y_\PP$ by $$\widetilde{\xi_u\boxtimes\xi_v}:=(e^{2\rho}\LL(\rho)\boxtimes\LL(\rho))_\PP\otimes\omega_{(X_u\times X_v)_\PP},$$ where $(e^{2\rho}\LL(\rho)\boxtimes\LL(\rho))_\PP$ is defined by $$E(T)_\PP\times^T\left(e^{2\rho}\LL(\rho)\boxtimes\LL(\rho)\right) \to Y_\PP.$$
\end{defn}

The diagonal map $\Delta: X\to Y$ gives rise to the embedding $\tilde{\Delta}: X_\PP\to Y_\PP$.  Thus, by the previous proposition, we may write

\begin{equation}
\label{eqcj}
\tilde{\Delta}_*[(\xi_w)_\PP]=\sum_{u,v\in W, j\in[N]^r} c^w_{u,v}(j){\hat{p}}^*[\OO_{\PP^j}]\cdot[\widetilde{\xi_u\boxtimes\xi_v}]\in K_0(Y_\PP),
\,\,\,\text{for unique}\,\, c^w_{u,v}(j)\in\Z,
\end{equation}
where $\hat{p}:Y_\PP\to\PP$ is the projection.

The following lemma makes precise the reduction from $T$-equivariant $K$-theory of $\X$ to the ordinary $K$-theory of the mixing space.  

\begin{lem}
\label{lemma1}
For any $u,v,w\in W$ we can choose $N$ large enough and express $$d^w_{u,v}=\sum_{j\in[N]^r} d^w_{u,v}(j)(e^{-\alpha_1}-1)^{j_1}\dots (e^{-\alpha_r}-1)^{j_r},$$ where $d^w_{u,v}(j)\in\Z$.

Then, $$d^w_{u,v}(j)=(-1)^{|j|} c^w_{u,v}(j),\,\,\,\text{where }|j|:=\sum_i^r j_i.$$
\end{lem}

\begin{proof}
This follows from Lemma \ref{coprodlem} and \cite[Lemma 6.2]{GK} (see also \cite[\S3]{AGM}).
\end{proof}

\begin{lem}
\label{lemtor}
For any coherent sheaf $\mathcal{S}$ on $\PP$ and any $u,v\in W$ we have:
\begin{enumerate}[a)]
\item $\Tor_i^{\OO_{\Y_\PP}}(\bar{p}^*(\mathcal{S}),\widetilde{\xi_u\boxtimes\xi_v})=0, \text{ }\forall i>0$,
\item $\Tor_i^{\OO_{\Y_\PP}}(\bar{p}^*(\mathcal{S}),\OO_{(X^u\times X^v)_\PP})=0, \text{ }\forall i>0$,
\end{enumerate}
where $\bar{p}:\Y_\PP\to\PP$ is the projection.
\end{lem}

\begin{proof}
As the statements are local in $\PP$, we may replace $\Y_{\PP}$ by $\mathcal{U} \times \Y$, for some open set $\mathcal{U} \subset \PP$.  Then,
$$\bar{p}^* \mathcal{S}\simeq S\boxtimes \OO_{\Y}$$
$$\widetilde{\xi_u\boxtimes\xi_v}\simeq\OO_{\mathcal{U}}\boxtimes(\xi_u\boxtimes\xi_v)$$
$$\OO_{(X^u\times X^v)_{\PP}}\simeq \OO_\mathcal{U}\boxtimes(\OO_{X^u}\boxtimes\OO_{X^v}) .$$\hspace{.004in}
Applying the above, followed by the Kunneth formula, gives, for part a), 
\begin{align*} 
\Tor_i^{\OO_{\mathcal{U}\times\Y}}(\bar{p}^*(\mathcal{S}),\widetilde{\xi_u\boxtimes\xi_v}) &=  \Tor_i^{\OO_{\mathcal{U}}\boxtimes\OO_{\Y}}(\mathcal{S}\boxtimes\OO_{\Y},\OO_\mathcal{U}\boxtimes(\xi_u\boxtimes\xi_v)) \\ 
&= \bigoplus_{j+k=i} \Tor_j^{\OO_\mathcal{U}}(\mathcal{S},\OO_\mathcal{U}) \otimes \Tor_k^{\OO_{\Y}}(\OO_{\Y},\xi_u \boxtimes \xi_v) \\
&= 0 ,\,\,\,\text{for}\,\, i>0.
\end{align*}
A similar computation gives part b).
\end{proof}

\begin{lem}
\label{lemtor2}
For any coherent sheaf $\mathcal{S}$ on $\PP$ and any $u\in W$ we have:
$$\Tor_i^{\OO_{Y_\PP}}({\hat{p}}^*(\mathcal{S}),\tilde{\Delta}_*((\xi_u)_\PP))=0, \text{ }\forall i>0 ,$$
where, as earlier,  $\hat{p}:Y_\PP\to\PP$ is the projection.
\end{lem}

\begin{proof}
As before, since the statement is local in $\PP$, we may replace $Y_{\PP}$ by $\mathcal{U} \times Y$, for some open set $\mathcal{U} \subset \PP$.  Then,
$${\hat{p}}^* S\simeq S\boxtimes \OO_{Y},$$
$$\tilde{\Delta}_*((\xi_u)_\PP)=\OO_{\mathcal{U}}\boxtimes\Delta_*(\xi_u).$$\hspace{.004in}
Now, proceed as in the proof of Lemma \ref{lemtor}.
\end{proof}

\begin{lem}
\label{lemcj} With notation as in (\ref{eqcj}) we have $$c^w_{u,v}(j)=\langle [\OO_{(X^u\times X^v)_\PP}],{\hat{p}}^*[\OO_{\PP_j}(-\partial\PP_j)]\cdot\tilde{\Delta}_*((\xi_w)_\PP) \rangle, $$ where the pairing $$\langle\text{ },\text{ }\rangle:K^0(\Y_\PP)\otimes K_0(Y_\PP)\to \Z$$ is defined (similar to (\ref{pairing})) by $$\langle[\mathcal{S}],[\mathcal{F}]\rangle=\sum_i(-1)^i\chi(\Y_\PP,\Tor^{\OO_{\Y_\PP}}_i(\mathcal{S},\mathcal{F})),$$  where $\chi$ denotes the Euler-Poincar\'{e} characteristic, and (as earlier) the map $\hat{p}:Y_\PP\to\PP$ denotes the projection.
\end{lem}

\begin{proof}
First, we have $$\langle [\OO_{(X^u\times X^v)_\PP}], {\hat{p}}^*[\OO_{\PP_j}(-\partial\PP_j)]\cdot\tilde{\Delta}_*((\xi_w)_\PP) \rangle=\langle \bar{p}^*[\OO_{\PP_j}(-\partial\PP_j)]\cdot[\OO_{(X^u\times X^v)_\PP}],\tilde{\Delta}_*((\xi_w)_\PP) \rangle ,$$ where $\bar{p}:\Y_\PP\to\PP$ denotes the projection.  To see this, first take a locally free finite resolution of $\OO_{(X^u\times X^v)_\PP}$ on a quasi-compact open subset of $\Y_\PP$ and a locally free finite resolution of $\OO_{\PP_j}(-\partial\PP_j)$ on $\PP$.  Then, use the fact that for a locally free sheaf $\mathcal{F}$ on a quasi-compact open subset of $\Y_\PP$ and a locally free sheaf $\mathcal{G}$ on $\PP$, we have $$\langle\mathcal{F},{\hat{p}}^*(\mathcal{G})\cdot\tilde{\Delta}_*((\xi_w)_\PP)\rangle=\langle \bar{p}^*(\mathcal{G})\cdot\mathcal{F}, \tilde{\Delta}_*((\xi_w)_\PP)\rangle .$$

Now, the lemma follows from (\ref{eqcj}), Lemma \ref{lemtor}, Proposition \ref{dual} and \cite[Identity 20]{K}.
\end{proof}

We now introduce the mixing group $\Gamma$ which acts on $\Y_\PP$. 

\begin{defn}(Mixing group $\Gamma$)
\label{mixgrp}
Let $T$ act on $B$ via $$t\cdot b=tbt^{-1}$$ for $t\in T, b\in B$.  This action induces a natural action of $\Delta T$ on $B\times B$.  Consider the ind-group scheme over $\PP$: $$(B^2)_\PP=E(T)_\PP\times^T B^2\to \PP.$$  
Let $\Gamma_0$ denote the group of global sections of $(B^2)_\PP$ under pointwise multiplication.  Since $GL(N+1)^r$ acts canonically on $(B^2)_\PP$ in a way compatible with its action on $\PP$, it also acts on $\Gamma_0$ via inverse pull-back.  We define the mixing group $\Gamma$ to be the semi direct product $\Gamma:=\Gamma_0\rtimes GL(N+1)^r$:

$$1\to\Gamma_0\to\Gamma\to GL(N+1)^r\to 1.$$
\end{defn}

By \cite[Lemmas 4.7 and 4.8 (more precisely, the paragraph following these lemmas)]{K},  we have the following two lemmas:

\begin{lem}
$\Gamma$ is connected. 
\end{lem} 

\begin{lem}\label{mixgroup}
For any $\bar{e}\in\PP$ and any $(b,b')$ in the fiber of $(B^2)_\PP$ over $\bar{e}$ there exists a section $\gamma\in\Gamma_0$ such that $\gamma(\bar{e})=(b,b')$.  
\end{lem}

We define the action of $\Gamma$ on $\Y_\PP$ via $$(\gamma,g)\cdot [e,(x,y)]=[ge,\gamma(ge)\cdot(x,y)]$$ for $\gamma\in\Gamma_0,g\in GL(N+1)^r,e\in E(T)_\PP,$ and $(x,y)\in \Y$, where the action of $\Gamma_0$ is via the standard action of $B\times B$ on $\Y=\X^2$.  It follows from Lemma \ref{mixgroup} that the orbits of the $\Gamma$-action on $\Y_\PP$ are precisely equal to $\{(C_w\times C_{w'})_\PP\}$.

\section{Statement of main results}

We now come to our main technical result.

\begin{thm}
\label{main}
For general $\gamma\in\Gamma$ and any $u,v,w\in W,j\in[N]^r$ we have:
\begin{enumerate}[a)]
\item For all $i>0$, $$\Tor_i^{\OO_{\Y_\PP}}\left(\gamma_*\OO_{(X^u\times X^v)_\PP},{\hat{p}}^*(\OO_{\PP_j}(-\partial\PP_j))\otimes_{\OO_{Y_\PP}}\tilde{\Delta}_*((\xi_w)_\PP)\right)=0 .$$
\item Assume $c^w_{u,v}(j) \neq 0$. For all $p\neq |j|+\ell(w)-\ell(u)-\ell(v)$, $$H^p\left(\Y_\PP,\gamma_*\OO_{(X^u\times X^v)_\PP}\otimes_{\OO_{\Y_\PP}}\left({\hat{p}}^*(\OO_{\PP_j}(-\partial\PP_j))\otimes_{\OO_{Y_\PP}}\tilde{\Delta}_*((\xi_w)_\PP)\right)\right)=0 ,$$ where $|j| :=\sum_{i=1}^r j_i$.
\end{enumerate}
\end{thm}

\begin{proof}

Deferred to the later sections.  Part a) is proved in Section $6$, while part b) is proved in Section $9$.

\end{proof}

Since $\Gamma$ is connected, Lemmas \ref{lemcj}, \ref{lemtor2}, and Theorem \ref{main} together give:

\begin{cor}
\label{maincor}
$(-1)^{\ell(w)-\ell(u)-\ell(v)+|j|}c^w_{u,v}(j)\in\Z_{\geq 0} .$
\end{cor}

As an immediate consequence of Corollary \ref{maincor} and Lemma \ref{lemma1} we get:

\begin{thm}
\label{positivity}
For any symmetrizable Kac-Moody group $G$ and any $u,v,w\in W$, the structure constants of the product in the basis $\{[\OO_{X^w}]\}$ in $K^0_T(\X)$ satisfy $$(-1)^{\ell(w)+\ell(u)+\ell(v)}d^w_{u,v}\in\Z_{\geq 0}[(e^{-\alpha_1}-1),\dots,(e^{-\alpha_r}-1)].$$
\end{thm}

\section{Proof of part a) of Theorem \ref{main}}

The key tool used to prove part a) of Theorem \ref{main} is the following transversality result, taken from \cite[Theorem 2.3]{AGM} (originally due to Sierra)

\begin{thm}
\label{trans}
Let $X$ be a variety with a left action of an algebraic group $G$ and let $\mathcal{F}$ be a coherent sheaf on $X$. Suppose that $\mathcal{F}$ is homologically transverse to the closures of the $G$-orbits on $X$. Then, for each coherent sheaf $\mathcal{E}$ on $X$,  there is a Zariski-dense open set $U \subseteq G$ such that $\Tor^{\OO_X}_i(\mathcal{F},g_*\mathcal{E}) = 0$ for all $i \geq 1$ and all $g \in U$.
\end{thm}

Theorem \ref{main} part a) is a particular case of the following slightly more general result by taking $\mathcal{E}={\hat{p}}^*(\OO_{\PP_j}(-\partial \PP_j))\otimes_{\OO_{Y_\PP}}\tilde{\Delta}_*((\xi_w)_\PP)$.  

\begin{thm}
\label{maintor}
Let $w\in W$ and let $\mathcal{E}$ be a coherent sheaf on $(Y_w)_\PP:=(X_w\times X_w)_\PP$.  Then, for general $\gamma\in\Gamma$ and any $u,v\in W$ we have:
$$\Tor_i^{\OO_{\Y_\PP}}\left(\gamma_*\OO_{(X^u\times X^v)_\PP},\mathcal{E}\right)=0,\text{ }\forall i>0.$$
\end{thm}

\begin{proof}
We first show that for any $\gamma \in \Gamma$, 
\begin{equation}
\label{eqtor=tor}
\Tor_i^{\OO_{\Y_\PP}}\left(\OO_{(X^u\times X^v)_\PP},\gamma_*\mathcal{E}\right)=\Tor_i^{\OO_{(Y_w)_\PP}}\left(\OO_{(Y_w)_\PP}\otimes_{\OO_{\Y_\PP}}\OO_{(X^u\times X^v)_\PP}, \gamma_*\mathcal{E}\right).
\end{equation}
Since $\gamma_*\mathcal{E}$ is a coherent sheaf on $(Y_w)_\PP$,  we can replace $\gamma_*\mathcal{E}$ by $\mathcal{E}$ itself. 

As the assertion is local on $\PP$ we may assume $\Y_\PP\simeq\PP\times \Y$ and that 

$$\OO_{(X^u\times X^v)_\PP}\simeq \OO_{\PP}\boxtimes(\OO_{X^u}\boxtimes\OO_{X^v}),$$

$$\OO_{(Y_w)_\PP}\simeq \OO_{\PP}\boxtimes(\OO_{X_w}\boxtimes\OO_{X_w}).$$\hspace{.004in}

To simplify notation let $A=\OO_{\Y_\PP}$, $B=\OO_{(Y_w)_\PP}$, $M=\OO_{(X^u\times X^v)_\PP}$ and $N=\mathcal{E}$.  Take an $A$-free resolution $\mathcal{F_\bullet}\to M$ on an open subset of $\Y_\PP$ of the form $(V^{u'}\times V^{v'})_\PP$, where $V^u$ is defined as in Lemma \ref{freeres}.  Then, the homology of the chain complex $B\otimes_A \mathcal{F_\bullet}$ is by definition $\Tor_\bullet^A(M,B)$.  Moreover,  $$\Tor^A_i(M,B)=0,\text{ }\forall i>0.$$   Indeed, locally on $\PP$, 

\begin{align*}
\Tor_i^A(M,B) &=\Tor_i^{\OO_\PP\boxtimes \OO_{\Y}}\left(\OO_\PP\boxtimes\OO_{(X^u\times X^v)},{\OO_\PP}\boxtimes\OO_{(X_w\times X_w)}\right)\\
&= \bigoplus_{j+k=i} \Tor_j^{\OO_\PP}(\OO_\PP,\OO_\PP) \otimes \Tor_k^{\OO_{\Y}}(\OO_{X^u\times X^v},\OO_{X_w\times X_w})\\
&=0,\text{ }\forall i>0
\end{align*}
by Lemma \ref{tor1} and  the Kunneth formula. Hence, $B\otimes_A \mathcal{F_\bullet}$ is a $B$-free resolution of $B\otimes_A M$.

Thus, the homology of the chain complex $N\otimes_B(B\otimes_A \mathcal{F}_\bullet)$ is equal to $\Tor_\bullet^B(B\otimes_A M,N)$; but, $$N\otimes_B(B\otimes_A \mathcal{F}_\bullet)=(N\otimes_B B)\otimes_A \mathcal{F}_\bullet=N \otimes_A \mathcal{F}_\bullet,$$
so the homology is also equal to $\Tor_\bullet^A(M,N)$.  Hence, $$\Tor_\bullet^B(B\otimes_A M,N)=\Tor_\bullet^A(M,N)$$ as desired.  This proves (\ref{eqtor=tor}).

Now, by Lemma \ref{mixgroup}, the closures of the $\Gamma$-orbits of $(Y_w)_\PP$ are precisely $(X_x\times X_y)_\PP$,  where $x,y\leq w$.  Equation (\ref{eqtor=tor}) implies that the sheaf $\mathcal{F}$ defined by $$\mathcal{F}:=\OO_{(Y_w)_\PP}\otimes_{\OO_{\Y_\PP}}\OO_{(X^u\times X^v)_\PP}$$ is homologically transverse to the $\Gamma$-orbit closures in $(Y_w)_\PP$.  Indeed, since $\OO_{(X_x\times X_y)_\PP}$ is a coherent $\OO_{(Y_w)_\PP}$-module when $x,y\leq w$, equation (\ref{eqtor=tor}) gives

\begin{align*}
\Tor_i^{\OO_{(Y_w)_\PP}}\left(\mathcal{F},\OO_{(X_x\times X_y)_\PP}\right) &= \Tor_i^{\OO_{\Y_\PP}}\left(\OO_{(X^u\times X^v)_\PP},\OO_{(X_x\times X_y)_\PP}\right)\\
&= 0,\text{ }\forall i>0
\end{align*}
by Lemma \ref{tor1} and the Kunneth formula.

Thus, by Theorem \ref{trans} (with $G=\Gamma,X=(Y_w)_\PP,$ and $\mathcal{E}$ and $\mathcal{F}$ as above), we conclude that for general $\gamma\in\Gamma$, 

\begin{equation}
\label{eqtor=0}
\Tor_i^{\OO_{(Y_w)_\PP}}\left(\OO_{(Y_w)_\PP}\otimes_{\OO_{\Y_\PP}}\OO_{(X^u\times X^v)_\PP},\gamma_*\mathcal{E}\right)=0,\text{ }\forall i>0.
\end{equation}
Here we note that although $\Gamma$ is infinite dimensional, the action of $\Gamma$ on $(Y_w)_\PP$ factors through the action of a finite dimensional quotient group of $\Gamma$.  

Now, (\ref{eqtor=tor}) gives 

$$\Tor_i^{\OO_{\Y_\PP}}\left(\OO_{(X^u\times X^v)_\PP},\gamma_*\mathcal{E}\right)=0,\text{ }\forall i>0,$$
which is equivalent to the desired vanishing.
\end{proof}

\section{The schemes $\ZZ$ and $\partial \ZZ$}

For $u,v\leq w$ we use the notation $X^{u,v}_w:=X^u_w\times X^v_w$.  We also write $X^2_w:=X_w\times X_w$.  For any $j\in[N]^r$, we let $(X_w)_j$ denote the inverse image of $\PP_j$ through the map $E(T)_\PP\times^T X_w\to\PP$. 

Similarly, for $u,v\leq w$ we write $Z^{u,v}_w:=Z^u_w\times Z^v_w$, where $Z^u_w$ is the $T$-equivariant desingularization of $X^u_w$ as in \cite[Theorem 6.8]{K}.  We also write $Z^2_w:=Z_w\times Z_w$, where $Z_w$ is a BSDH variety as in \cite[\S7.1.3]{Kbook}.  For any $j\in[N]^r$, we let $(Z_w)_j$ denote the inverse image of $\PP_j$ through the map $E(T)_\PP\times^T Z_w\to\PP$. 

The action of $B$ on $Z_w$ factors through the action of a finite dimensional quotient group $\bar{B}$ containing the maximal torus $H$.  Further, the action of $\Gamma$ on $(X^2_w)_\PP$ descends to an action of the finite dimensional quotient group $$\bar{\Gamma}:=\bar{\Gamma}_0\rtimes GL(N+1)^r ,$$ where $\bar{\Gamma}_0$ is the group of global sections of the bundle $E(T)_\PP\times^T(\bar{B}^2)\to\PP$.  

From \cite[Lemmas 6.11 and 6.12]{K} we have

\begin{lem}
\label{mflat}
Let $u,v\leq w$.  The map $$m:\bar{\Gamma}\times (X^{u,v}_w)_\PP\to(X^2_w)_\PP,\,\,\,m(\gamma,x)=\gamma\cdot \pi_2(x)$$ is flat, where $\pi_2: (X^u_w\times X^v_w)_\PP\to (X^2_w)_\PP$ is induced from the canonical map $X^u_w\times X^v_w\to X^2_w$.  
\end{lem}

\begin{lem}
\label{msmooth}
Let $u,v\leq w$.  The map $$\tilde{m}:\bar{\Gamma}\times (Z^{u,v}_w)_\PP\to(Z^2_w)_\PP,\,\,\,\tilde{m}(\gamma,x)=\gamma\cdot \tilde{\pi}_2(x)$$ is smooth, where $\tilde{\pi}_2: (Z^u_w\times Z^v_w)_\PP\to (Z^2_w)_\PP$ is induced from the canonical map $Z^u_w\times Z^v_w\to Z^2_w$. 
\end{lem}

Define $\ZZ$ to be the fiber product $$\ZZ:=\left(\bar{\Gamma}\times(X^{u,v}_w)_\PP\right)\times_{(X^2_w)_\PP}\tilde{\Delta}((X_w)_j)$$ and $\widetilde{\ZZ}$ to be the fiber product $$\widetilde{\ZZ}:=\left(\bar{\Gamma}\times(Z^{u,v}_w)_\PP\right)\times_{(Z^2_w)_\PP}\tilde{\Delta}((Z_w)_j)$$ as in the commutative diagram:


\[
\begin{tikzcd}[,column sep=small]
& & \widetilde{\ZZ} \arrow[llddd,"\tilde{\pi}",swap] \arrow[dddddd, bend right = 90, "f", swap] \arrow[hook]{dd}{\tilde{\iota}} \arrow["\text{(smooth)}"']{rr}{\tilde{\mu}} & & \tilde{\Delta}((Z_w)_j) \arrow[hook]{dd}{}\\
& & & \square & \\
& & \bar{\Gamma}\times(Z^{u,v}_w)_\PP \arrow{dd}{\theta} \arrow["\text{(smooth)}"']{rr}{\tilde{m}} & & (Z^2_w)_\PP \arrow{dd}{\beta}\\
\bar{\Gamma} & & & & \\
& & \bar{\Gamma}\times(X^{u,v}_w)_\PP \arrow["\text{(flat)}"']{rr}{m} & & (X^2_w)_\PP \\
& & & \square & \\
& & \ZZ \arrow{lluuu}{\pi} \arrow["\text{(flat)}"']{rr}{\mu} \arrow[swap,hook]{uu}{i} & & \tilde{\Delta}((X_w)_j) \arrow[hook]{uu}{}
\end{tikzcd}
\]\\

In the above diagram, $\square$ denotes a fiber square.  Note that the maps $\theta$ and $\beta$ above are desingularizations.  The maps $\pi:\ZZ\to\bar{\Gamma}$ and $\tilde{\pi}:\widetilde{\ZZ}\to\bar{\Gamma}$ are induced by the projections onto the first factor.  The map $f:\widetilde{\ZZ}\to\ZZ$ is defined by $f:=\theta \circ \tilde{\iota}$.  It is clear that the image of $f$ is indeed contained inside of $\ZZ$ using commutativity of the diagram, along with the fact that $\beta(\tilde{\Delta}((Z_w)_j))=\tilde{\Delta}((X_w)_j))$.  

We define the boundary of $(X_w)_j$ by $$\partial ((X_w)_j):=(\partial X_w)_j\cup(X_w)_{\partial\PP_j}$$ and similarly define the boundary of $(Z_w)_j$ by $$\partial ((Z_w)_j):=(\partial Z_w)_j\cup(Z_w)_{\partial\PP_j},$$ where $\partial Z_w:=\varphi^{-1}(\partial X_w)$ and $\varphi: Z_w\to X_w$ denotes the desingularization. 

We define the boundary of $\ZZ$ by $$\partial \ZZ:=\left(\bar{\Gamma}\times(X^{u,v}_w)_{\PP}\right)\times_{(X^2_w)_\PP}\tilde{\Delta}(\partial ((X_w)_j))$$ and similarly define the boundary of $\tilde{\ZZ}$ by $$\partial \tilde{\ZZ}:=\left(\bar{\Gamma}\times(Z^{u,v}_w)_{\PP}\right)\times_{(Z^2_w)_\PP}\tilde{\Delta}(\partial ((Z_w)_j)).$$  Observe that $f^{-1}(\partial\ZZ)=\partial\widetilde{\ZZ}$ is the scheme-theoretic inverse image.  

We will need the following lemmas, which are restatements of \cite[Lemmas 7.2 and 7.3]{K} respectively (\cite[Lemma 7.3]{K} is originally from \cite[Lemma on page 108]{FP}).

\begin{lem}
\label{CMlem1}
Let $f:W\to X$ be a flat morphism from a pure-dimensional CM scheme $W$ of finite type over $\C$ to a CM irreducible variety $X$ and let $Y$ be a closed CM subscheme of $X$ of pure codimension d.  Set $Z:=f^{-1}(Y)$.  If $\codim_W(Z)\geq d$ then equality holds and $Z$ is CM.  
\end{lem}

\begin{lem}
\label{CMlem2}
Let $f:W\to X$ be a morphism from a pure-dimensional CM scheme $W$ of finite type over $\C$ to a smooth irreducible variety $X$ and let $Y$ be a closed CM subscheme of $X$ of pure codimension d.  Set $Z:=f^{-1}(Y)$.  If $\codim_W(Z)\geq d$ then equality holds and $Z$ is CM.  
\end{lem}

\begin{prop}
\label{Z}
The scheme $\ZZ$ is normal, irreducible, and CM, of dimension $$\dim\ZZ=|j|+\ell(w)-\ell(u)-\ell(v)+\dim\bar{\Gamma}.$$

The scheme $\widetilde{\ZZ}$ is irreducible, and the map $f:\widetilde{\ZZ}\to\ZZ$ is a proper birational map.  Hence, the scheme $\widetilde{\ZZ}$ is a desingularization of $\ZZ$. 
\end{prop}

\begin{proof}
By \cite[Chapter III, Corollary 9.6]{H}, the fibers of $m$ are pure of dimension $$\dim \bar{\Gamma}+\dim (X^{u,v}_w)_\PP - \dim (X^2_w)_\PP .$$  Since the fibers of $\mu$ are the same as those of $m$, applying loc. cit.  to $\mu$ gives that $\ZZ$ is pure of dimension $$\dim \ZZ=\dim \bar{\Gamma}+\dim (X^{u,v}_w)_\PP - \dim (X^2_w)_\PP + \dim \tilde{\Delta}((X_w)_j)$$ $$=|j|+\ell(w)-\ell(u)-\ell(v)+\dim\bar{\Gamma}.$$

The remainder of the proof of the proposition follows from the proof of \cite[Proposition 7.4 and Lemma 7.5]{K} (Note that there is a slight difference in the definition of $\ZZ$ and $\widetilde{\ZZ}$ between ours and that in \cite{K}.)  Further, the scheme $\widetilde{ZZ}$ is non-singular, since $\tilde{\mu}$ is a smooth morphism with non-singular base.  
\end{proof}

\begin{lem}
\label{lemact}
Let $G$ be a group acting on a set $X$ and let $Y\subset X$.  Consider the action map $m:G\times Y\to X$.  For $x\in X$ denote the orbit of $x$ by $O(x)$ and the stabilizer by $\text{Stab}(x)$.  Then, $\text{Stab}(x)$ acts on the fiber $m^{-1}(x)$ and $\text{Stab}(x)\backslash m^{-1}(x)\simeq O(x)\cap Y$.
\end{lem}

\begin{proof}
It is easy to check that $$m^{-1}(x)=\left\{(g,h^{-1}x):h\in G, \text{ } h^{-1}x\in  Y,\text{ } g\in\text{Stab}(x)\cdot h\right\}.$$
Thus, $\text{Stab}(x)$ acts on $m^{-1}(x)$ by left multiplication on the left component.  Since every element of $O(x)\cap Y$ is of the form $h^{-1}x$ for some $h\in G$,  the second projection $m^{-1}(x)\to O(x)\cap Y$ is surjective.  This map clearly factors through the quotient to give a map $\text{Stab}(x)\backslash m^{-1}(x)\to O(x)\cap Y$.  To show this induced map is injective, note first that each class has a representative of the form $(h,h^{-1}x)$.  Now, if $(h_1,h_1^{-1}x)$ and $(h_2,h_2^{-1}x)$ satisfy $h_1^{-1}x=h_2^{-1}x$ then $h_2 h_1^{-1} x = x$, i.e. $h_2 h_1^{-1}\in \text{Stab}(x)$, i.e. $h_2 \in \text{Stab}(x)\cdot h_1$, i.e. $(h_1,h_1^{-1}x)$ and $(h_2,h_2^{-1}x)$ belong to the same class.
\end{proof}

\begin{prop}
\label{ratlsing}
The scheme $\ZZ$ has rational singularities.  
\end{prop}

\begin{proof}  Since $\mu$ is flat and $\tilde{\Delta}((X_w)_j)$ has rational singularities, by \cite[Th\'eor\`em 5]{Elk} it is sufficient to show that the fibers of $\mu$ are disjoint unions of irreducible varieties with rational singularities (in fact, the fibers of $\mu$ are irreducible, but we do not provide a proof here as we do not need this fact).  

Let $x\in \tilde{\Delta}((C_{w'})_j)$,  where $w'\leq w$.  Then, by Lemmas \ref{lemact} and \ref{mixgroup}, we have $\text{Stab}(x)\backslash\mu^{-1}(x)\simeq (X^u\cap C_{w'}\times X^v\cap C_{w'})_\PP$, where $\text{Stab}(x)$ is taken with respect to the action of $\bar{\Gamma}$ on $(X^2_w)_\PP$.  By \cite[Proposition 3, \S2.5]{Ser},  the quotient map $\bar{\Gamma}\to\text{Stab}(x)\backslash\bar{\Gamma}$ is locally trivial in the \'{e}tale topology.  

Consider the pullback diagram:

\[
\begin{tikzcd}[column sep=small, row sep=huge]
\mu^{-1}(x) \arrow[]{d}{} & \subseteq & \bar{\Gamma}\times (X^{u,v}_w)_\PP \arrow[]{d}{}\\
\text{Stab}(x)\backslash\mu^{-1}(x) & \subseteq & \left(\text{Stab}(x)\backslash\bar{\Gamma}\right)\times (X^{u,v}_w)_\PP
\end{tikzcd}
\]
Since the right vertical map is a locally trivial fibration in the \'{e}tale topology, the left vertical map is too.  Now, $\text{Stab}(x)\backslash\mu^{-1}(x)\simeq (X^u\cap C_{w'}\times X^v\cap C_{w'})_\PP$ has rational singularities by \cite[Theorem 3.1]{KSch}.  Further, $\text{Stab}(x)$ being smooth and $\mu^{-1}(x)\to\text{Stab}(x)\backslash\mu^{-1}(x)$ being locally trivial in the \'{e}tale topology, we get that $\mu^{-1}(x)$ is a disjoint union of irreducible varieties with rational singularities by \cite[Corollary 5.11]{KM}.

\end{proof}

\begin{prop}
\label{dZCM}
The scheme $\partial \ZZ$ is pure of codimension $1$ in $\ZZ$ and is CM.  
\end{prop}

\begin{proof}
Using the fact that $\tilde{\Delta}((X_w)_{\partial\PP_j})$ and $\tilde{\Delta}((\partial X_w)_j)$ are equidimensional, applying \cite[Chapter III, Corollary 9.6]{H} to their irreducible components gives that $\left(\bar{\Gamma}\times(X^{u,v}_w)_{\PP}\right)\times_{(X^2_w)_\PP}\tilde{\Delta}((X_w)_{\partial\PP_j})$ and $\left(\bar{\Gamma}\times(X^{u,v}_w)_{\PP}\right)\times_{(X^2_w)_\PP}\tilde{\Delta}((\partial X_w)_j)$ are both pure of dimension $$\dim \bar{\Gamma}+\dim (X^{u,v}_w)_\PP - \dim (X^2_w)_\PP + \dim \tilde{\Delta}((X_w)_j) -1.$$  Hence, $\partial \ZZ$ is pure of codimension $1$ in $\ZZ$.  

A similar argument also gives that $\left(\bar{\Gamma}\times(X^{u,v}_w)_{\PP}\right)\times_{(X^2_w)_\PP}\tilde{\Delta}((\partial X_w)_{\partial \PP_j})$ is pure of dimension $$\dim \bar{\Gamma}+\dim (X^{u,v}_w)_\PP - \dim (X^2_w)_\PP + \dim \tilde{\Delta}((X_w)_j) -2.$$

Next we show that $\partial \ZZ$ is CM.  Since $(X_w)_j$ is a locally trivial fibration over $\PP_j$, it is locally a product of CM schemes and hence is CM.  Also, since $\partial\PP_j$ is CM, we similarly have that $(X_w)_{\partial\PP_j}$ and hence $\tilde{\Delta}((X_w)_{\partial\PP_j})$ is CM.  Now, applying Lemma \ref{CMlem1} to $\mu:\ZZ\to\tilde{\Delta}((X_w)_j)$ gives that $\left(\bar{\Gamma}\times(X^{u,v}_w)_{\PP}\right)\times_{(X^2_w)_\PP}\tilde{\Delta}((X_w)_{\partial\PP_j})$ is CM, since $\tilde{\Delta}((X_w)_{\partial\PP_j})$ and $\mu^{-1}(\tilde{\Delta}((X_w)_{\partial\PP_j}))=\left(\bar{\Gamma}\times(X^{u,v}_w)_{\PP}\right)\times_{(X^2_w)_\PP}\tilde{\Delta}((X_w)_{\partial\PP_j})$ are of pure codimension $1$ in $\tilde{\Delta}((X_w)_{j})$ and $\ZZ$ respectively.  

Observe that $\partial X_w$ is CM.  To prove this, use Proposition \ref{xiomega} and the argument as in the proof of \cite[Corollary 10.5]{K} by taking an embedding of $X_w$ into a smooth projective variety.  Thus, $(\partial X_w)_j$ is locally a product of CM schemes and hence is CM, and hence so is $\tilde{\Delta}((\partial X_w)_j)$.  Now, Lemma \ref{CMlem1} applied to $\mu:\ZZ\to\tilde{\Delta}((X_w)_j)$ gives that $$\left(\bar{\Gamma}\times(X^{u,v}_w)_{\PP}\right)\times_{(X^2_w)_\PP}\tilde{\Delta}((\partial X_w)_j)$$ is CM since $\tilde{\Delta}((\partial X_w)_j)$ and $\mu^{-1}(\tilde{\Delta}((\partial X_w)_j))=\left(\bar{\Gamma}\times(X^{u,v}_w)_{\PP}\right)\times_{(X^2_w)_\PP}\tilde{\Delta}((\partial X_w)_j)$ are of pure codimension $1$ in $\tilde{\Delta}((X_w)_j)$ and $\ZZ$ respectively.  

The intersection $$\left(\bar{\Gamma}\times(X^{u,v}_w)_{\PP}\right)\times_{(X^2_w)_\PP}\tilde{\Delta}((\partial X_w)_{\partial\PP_j})$$ $$=\left(\left(\bar{\Gamma}\times(X^{u,v}_w)_{\PP}\right)\times_{(X^2_w)_\PP}\tilde{\Delta}((X_w)_{\partial\PP_j})\right)\cap\left(\left(\bar{\Gamma}\times(X^{u,v}_w)_{\PP}\right)\times_{(X^2_w)_\PP}\tilde{\Delta}((\partial X_w)_j)\right)$$ is of pure codimension $1$ in both $\left(\bar{\Gamma}\times(X^{u,v}_w)_{\PP}\right)\times_{(X^2_w)_\PP}\tilde{\Delta}((X_w)_{\partial\PP_j})$ and $\left(\bar{\Gamma}\times(X^{u,v}_w)_{\PP}\right)\times_{(X^2_w)_\PP}\tilde{\Delta}((\partial X_w)_j)$.  Now, \cite[Exercise 18.13]{E} gives that the union $\partial \ZZ$ is CM iff the intersection $\left(\bar{\Gamma}\times(X^{u,v}_w)_{\PP}\right)\times_{(X^2_w)_\PP}\tilde{\Delta}((\partial X_w)_{\partial\PP_j})$ is.  But Lemma \ref{CMlem1} applied to $\mu:\ZZ\to\tilde{\Delta}((X_w)_j)$ gives that $\left(\bar{\Gamma}\times(X^{u,v}_w)_{\PP}\right)\times_{(X^2_w)_\PP}\tilde{\Delta}((\partial X_w)_{\partial\PP_j})$ is CM, since $\tilde{\Delta}((\partial X_w)_{\partial\PP_j})$ and $\mu^{-1}(\tilde{\Delta}((\partial X_w)_{\partial\PP_j}))=\left(\bar{\Gamma}\times(X^{u,v}_w)_{\PP}\right)\times_{(X^2_w)_\PP}\tilde{\Delta}((\partial X_w)_{\partial\PP_j})$ are of pure codimension $2$ in $\tilde{\Delta}((X_w)_j)$ and $\ZZ$ respectively.    
\end{proof}

\begin{lem}
\label{affine}
The morphism $\mu:\ZZ\to \tilde{\Delta}((X_w)_j)$ is affine.
\end{lem}

\begin{proof}
Since $\tilde{\Delta}((X_w)_j)$ is a closed subscheme of $(X^2_w)_\PP$ it suffices to show that the map $m: \bar{\Gamma}\times(X^{u,v}_w)_\PP\to(X^2_w)_\PP$ is an affine morphism.  Now, if $f:X\to Y$ is an affine morphism and $Z\subset X$ is a closed subscheme then $f|_{Z}:Z\to Y$ is clearly an affine morphism.  Thus, it suffices to show that $\bar{\Gamma}\times(X^2_w)_\PP\to(X^2_w)_\PP$ is an affine morphism.  Further, if $X$ and $Y$ are total spaces of principal $T$-bundles and $f:X\to Y$ is a $T$-equivariant map,  then $f$ is affine iff $\bar{f}:X/T\to Y/T$ is affine.  Thus, it suffices to show that $$\hat{\mu}:\bar{\Gamma}\times (E(T)_\PP\times X^2_w)\to E(T)_\PP\times X^2_w$$ is affine.  Recall that $\bar{\Gamma}=\bar{\Gamma}_0\rtimes GL(N+1)^r$ and  $\hat{\mu}$ is given by $\hat{\mu}((\sigma,g),(e,x))=(ge,\sigma(ge)\cdot x)$, where $\sigma\in\bar{\Gamma}_0, g\in GL(N+1)^r, e\in E(T)_\PP, x\in X^2_w$.  Write $\hat{\mu}$ as a composite $\hat{\mu}=\mu_3\circ\mu_2\circ\mu_1$ where: 
$$\mu_1:\bar{\Gamma}_0\times GL(N+1)^r\times E(T)_\PP\times X^2_w\to \bar{\Gamma}_0 \times E(T)_\PP\times X^2_w,\text{ }(\sigma,g,e,x)\mapsto(\sigma,g\cdot e, x)$$
$$\mu_2:\bar{\Gamma}_0\times E(T)_\PP\times X^2_w\to \bar{B}^2 \times E(T)_\PP\times X^2_w,\text{ }(\sigma, e, x)\mapsto (\sigma(e), e, x)$$
$$\mu_3:\bar{B}^2 \times E(T)_\PP\times X^2_w\to E(T)_\PP\times X^2_w,\text{ }((b_1,b_2),e,(x,y))\mapsto(e,b_1 x, b_2 y).$$
\vspace{.004in}
As the composition of two affine morphisms is affine, it suffices to prove that $\mu_1,\mu_2,\mu_3$ are affine.  Moreover, if $f:X\to Y$ is affine then $f\times\text{Id}_Z:X\times Z \to Y\times Z$ is affine for any scheme $Z$.  Hence, it suffices to prove that the following maps $\hat{\mu}_1,\hat{\mu}_2,$ and $\hat{\mu}_3$ are affine:
$$\hat{\mu}_1:GL(N+1)^r\times E(T)_\PP\to E(T)_\PP,\text{ }(g,e)\mapsto g\cdot e $$
$$\hat{\mu}_2:\bar{\Gamma}_0\times E(T)_\PP\to \bar{B}^2\times E(T)_\PP,\text{ }(\sigma, e)\mapsto(\sigma(e),e)$$
$$\hat{\mu}_3:\bar{B}^2\times X^2_w\to X^2_w,\text{ }((b_1,b_2),(x,y))\mapsto(b_1x,b_2y).$$
\begin{enumerate}
\item \textbf{$\hat{\mu}_1$ is affine}:  Since $E(T)_\PP=(\C^{N+1}\setminus\{0\})^r$, it suffices to prove that $\theta: GL(N+1)\times (\C^{N+1}\setminus\{0\})\to\C^{N+1}\setminus\{0\}$, $(g,v)\mapsto g\cdot v$ is affine.  Now, consider the map $\bar{\theta}: GL(N+1)\times \C^{N+1}\to \C^{N+1}$, $(g,v)\mapsto g\cdot v$.  Since both the domain and codomain are affine, $\bar{\theta}$ is an affine morphism.  Moreover, $\theta=\bar{\theta}|_{\bar{\theta}^{-1}(\C^{N+1}\setminus\{0\})}$.  Thus, $\theta$ is affine.  

\item \textbf{$\hat{\mu}_2$ is affine}:  Take an affine open subset $\mathcal{U}\subset E(T)_\PP$.  Then, $\bar{B}^2\times\mathcal{U}$ is an affine open subset in $\bar{B}^2\times E(T)_\PP$.  Now, $\hat{\mu}_2^{-1}(\bar{B}^2\times\mathcal{U})=\bar{\Gamma}_0\times \mathcal{U}$.  Since $\bar{\Gamma}_0$ is affine, so is $\bar{\Gamma}_0\times \mathcal{U}$.  Thus, $\hat{\mu}_2$ is affine.  

\item \textbf{$\hat{\mu}_3$ is affine}: It suffices to prove that $\delta:\bar{B}\times X_w\to X_w$, $(b,x)\mapsto bx$ is affine.  Take an affine open subset $V\subset X_w$.  Then, $\delta^{-1}(V)=\bigcup_{b\in\bar{B}}(b,b^{-1}V)$.  Consider the scheme isomorphism $i:\bar{B}\times X_w\mapsto \bar{B}\times X_w$, $(b,x)\mapsto (b,b\cdot x)$.  Then, $i(\delta^{-1}(V))=\bar{B}\times V$.  But, since $\bar{B}\times V$ is affine, so is $\delta^{-1}(V)$.  Thus, $\delta$ is an affine morphism and hence so is $\hat{\mu}_3$. 
\end{enumerate} 
\end{proof}

Let $\pi:\ZZ\to \bar{\Gamma}$ denote the projection onto the first factor and $\pi_1:\partial \ZZ\to \bar{\Gamma}$ denote the restriction of $\pi$ to $\partial \ZZ$.  We define the fibers 
\begin{equation}
\label{eqNdef}
N_\gamma:=\pi^{-1}(\gamma)\simeq\gamma((X_w^{u,v})_\PP)\cap\tilde{\Delta}((X_w)_j)
\end{equation}
and 
\begin{equation}
\label{eqMdef}
M_\gamma:=\pi_1^{-1}(\gamma)\simeq\gamma((X_w^{u,v})_{\PP})\cap\tilde{\Delta}(\partial ((X_w)_{j})).
\end{equation}

\begin{cor}
\label{corNM1}
Assume that $c^w_{u,v}(j)\neq 0$, where $c^w_{u,v}(j)$ are defined by (\ref{eqcj}).  Then, for general $\gamma\in\bar{\Gamma}$, we have that $N_\gamma$ (defined by (\ref{eqNdef})) is CM of pure dimension.  In fact, for $\gamma\in\bar{\Gamma}$ such that $$\dim N_\gamma=\dim \ZZ - \dim \bar{\Gamma}=|j|+\ell(w)-\ell(u)-\ell(v),$$ $N_\gamma$ is CM, and this condition is satisfied for general $\gamma$.  

Similarly, if $|j|+\ell(w)-\ell(u)-\ell(v)>0$, for general $\gamma\in\bar{\Gamma}$, we have that $M_\gamma$ (defined by (\ref{eqMdef})) is CM of pure codimension $1$ in $N_\gamma$.  If $|j|+\ell(w)-\ell(u)-\ell(v)=0$, for general $\gamma\in\bar{\Gamma}$, we have that $M_\gamma$ is empty.  
\end{cor}

\begin{proof}
First we show that $\pi$ is surjective when $c^w_{u,v}(j)\neq 0$.  From the definition of $\pi$ we have that $$\text{Im }\pi=\left\{ \gamma\in\bar{\Gamma}:\gamma((X^{u,v}_w)_\PP)\cap\tilde{\Delta}((X_w)_j)\neq\emptyset \right\}.$$  Since $\bar{\Gamma}$ is connected, Lemma \ref{lemcj} along with the assumption $c^w_{u,v}(j)\neq 0$ gives that $\gamma((X^{u,v})_\PP)\cap\tilde{\Delta}((X_w)_j)\neq\emptyset$ for any $\gamma\in\bar{\Gamma}$.  Since $\gamma((X^{u,v})_\PP)\cap\tilde{\Delta}((X_w)_j)=\gamma((X^{u,v}_w)_\PP)\cap\tilde{\Delta}((X_w)_j)$, we get that $\pi$ is surjective.  

Now, since $\ZZ$ is CM, applying Lemma \ref{CMlem2} to $\pi$ gives that if 
\begin{equation}
\label{eqN}
\codim_\ZZ(N_\gamma)=\dim\bar{\Gamma} ,
\end{equation}
then $N_\gamma$ is CM.  By \cite[Chapter I, \S6.3, Theorem 1.25]{S}, this condition holds for $\gamma$ in a dense open subset of $\bar{\Gamma}$.  Thus, $N_\gamma$ is CM for $\gamma$ satisfying $\dim N_\gamma=\dim \ZZ - \dim \bar{\Gamma}$, and this condition is satisfied for general $\gamma$.  

Next, we show that if $|j|+\ell(w)-\ell(u)-\ell(v)>0$ then $\pi_1$ is surjective.  First note that since $\pi_1$ is projective, if it is not surjective, its image is a proper closed subset of $\bar{\Gamma}$.  Thus, for general $\gamma\in\bar{\Gamma}$, $M_\gamma=\emptyset$, i.e., $N_\gamma\subset\ZZ\setminus\partial\ZZ$.  As $\mu$ is an affine morphism by Lemma \ref{affine}, $\ZZ\setminus\partial\ZZ$ is affine.   But, $N_\gamma$ is projective of positive dimension (since $|j|+\ell(w)-\ell(u)-\ell(v)>0$) which gives a contradiction.  

In particular, there is at least one irreducible component of $\partial\ZZ$ on which $\pi_1$ is surjective.  The other irreducible components are mapped to closed subsets of $\bar{\Gamma}$.  Let $\mathcal{U}$ be the complement of the union of the images of the irreducible components on which $\pi_1$ is not surjective.  Then, $\mathcal{U}$ is open, since there are only finitely many irreducible components.  Applying \cite[Chapter I, \S6.3,  Theorem 1.25]{S} to $\pi_1$ on each irreducible component of $\partial \ZZ$ which surjects onto $\bar{\Gamma}$ and then intersecting with $\mathcal{U}$ gives that for general $\gamma\in\bar{\Gamma}$
\begin{equation}
\label{eqM}
\codim_{\partial \ZZ}(M_\gamma)=\dim\bar{\Gamma}.
\end{equation}
Thus, by Lemma \ref{CMlem2}, $M_\gamma$ is CM.  Moreover, (\ref{eqN}) and (\ref{eqM}) together imply that $M_\gamma$ is pure of codimension $1$ in $N_\gamma$.  

In the case where $|j|+\ell(w)-\ell(u)-\ell(v)=0$, we have $\dim \ZZ=\dim \bar{\Gamma}$, so that $\dim \partial \ZZ < \dim \bar{\Gamma}$.  Thus, $\text{Im } \pi_1$ is a proper closed subset of $\bar{\Gamma}$, and hence $M_\gamma$ is empty for general $\gamma$.  
\end{proof}

\begin{cor}
\label{corNM2}
Assume that $c^w_{u,v}(j)\neq 0$.  Then, for general $\gamma\in\bar{\Gamma}$, we have $$\Ext^i_{\OO_{N_\gamma}}(\OO_{N_\gamma}(-M_\gamma),\omega_{N_\gamma})=0, \text{ }\forall i>0.$$
\end{cor}

\begin{proof}
By \cite[Proposition 11.33 and Corollary 11.43]{I} we have $$\Ext^i_{\OO_{N_\gamma}}(\OO_{M_\gamma},\omega_{N_\gamma})=0,\text{ unless }i=1,$$ and, of course, $$\Ext^i_{\OO_{N_\gamma}}(\OO_{N_\gamma},\omega_{N_\gamma})=0, \text{ }\forall i>0.$$  Hence, the desired result follows from the long exact $\Ext$ sequence associated to the short exact sequence $$0\to \OO_{N_\gamma}(-M_\gamma) \to \OO_{N_\gamma} \to \OO_{M_\gamma} \to 0.$$
\end{proof}

\section{Application of Kawamata-Viehweg Vanishing}

In this section we assume that $c^w_{u,v}(j)\neq 0$.  

Let $f:X\to Y$ be a proper morphism between schemes, with $X$ irreducible, and let $\mathcal{M}$ be a line bundle on $X$.  Then,  $\mathcal{M}$ is said to be \emph{$f$-nef} if it has nonnegative intersection with every curve contained in a fiber of $f$.  The line bundle $\mathcal{M}$ is said to be \emph{$f$-big} if $\text{rank }f_*(\mathcal{M}^k)>c\cdot k^n$ for some $c>0$ and $k\gg 1$, where $n$ is the dimension of a general fiber of $f$.  A Weil divisor $D\subset X$ has \emph{normal crossings} if all of its irreducible components intersect transversely.  It is easy to see that $\partial Z_w$ and hence $\partial ((Z_w)_j)$ are normal crossings divisors in $Z_w$ and $(Z_w)_j$ respectively.  

We will need the following relative Kawamata-Viehweg vanishing theorem from \cite[Theorem 2.4]{AGM}, which was originally extracted from \cite[Corollary 6.11]{EV}.

\begin{thm}
\label{KV}
Let $f:\widetilde{Z}\to Z$ be a proper surjective morphism of varieties, with $\widetilde{Z}$ nonsingular.  Let $\mathcal{M}$ be a line bundle on $\widetilde{Z}$ such that $\mathcal{M}^{\bar{N}}(-D)$ is $f$-nef and $f$-big for a normal crossing divisor $D=\sum_{j=1}^r a_jD_j$, where $0<a_j<\bar{N}$ for all $j$.  Then, $$R^if_*(\mathcal{M}\otimes\omega_{\widetilde{Z}})=0,\text{ }\forall i>0.$$  
\end{thm}

\begin{defn}
We define the sheaf $$\omega_\ZZ(\partial \ZZ):=\Hom_{\OO_\ZZ}(\OO_\ZZ(-\partial\ZZ),\omega_\ZZ).$$
\end{defn}

\begin{thm}
\label{thmf_*}
We have $f_*\omega_{\widetilde{\ZZ}}(\partial\widetilde{\ZZ})=\omega_\ZZ(\partial\ZZ)$, where $\widetilde{\ZZ},\ZZ,\partial\widetilde{\ZZ},\partial\ZZ$ and the morphism $f:\widetilde{\ZZ}\to\ZZ$ are defined after Lemma \ref{msmooth}. (Observe that since $\ZZ$ is CM, the dualizing sheaf $\omega_\ZZ$ makes sense.)
\end{thm}

\begin{proof}
First, we claim 

\begin{equation}
\label{pullbackhom}
\OO_{\widetilde{\ZZ}}(\partial\widetilde{\ZZ})\simeq \Hom_{\OO_{\widetilde{\ZZ}}}(f^*\OO_\ZZ(-\partial\ZZ),\OO_{\widetilde{\ZZ}}),
\end{equation}
where $\OO_{\widetilde{\ZZ}} (\partial\widetilde{\ZZ}):= \Hom_{\OO_{\widetilde{\ZZ}}}(\OO_{\widetilde{\ZZ}}(-\partial\widetilde{\ZZ}),\OO_{\widetilde{\ZZ}}).$ To see this, first note that since $f^{-1}(\partial\ZZ)=\partial\widetilde{\ZZ}$ is the scheme-theoretic inverse image (cf. $\S$6), the natural  morphism $f^*\OO_\ZZ(-\partial\ZZ)\to\OO_{\widetilde{\ZZ}}(-\partial \widetilde{\ZZ})$ is surjective (cf., e.g. \cite[\href{http://stacks.math.columbia.edu/tag/01HJ}{Tag 01HJ} Lemma 25.4.7]{stacks-project}).  As $f$ is a desingularization (cf. Proposition \ref{Z}), the kernel of this morphism is supported on a proper closed subset of $\widetilde{\ZZ}$ and hence is a torsion sheaf.  This implies that the dual map $\OO_{\widetilde{\ZZ}}(\partial\widetilde{\ZZ})\to \Hom_{\OO_{\widetilde{\ZZ}}}(f^*\OO_\ZZ(-\partial\ZZ),\OO_{\widetilde{\ZZ}})$ is an isomorphism, proving (\ref{pullbackhom}). 

To complete the proof of the theorem, we compute:

\begin{align*}
f_*(\omega_{\widetilde{\ZZ}}\otimes \OO_{\widetilde{\ZZ}}(\partial\widetilde{\ZZ})) &= f_*(\omega_{\widetilde{\ZZ}}\otimes \Hom_{\OO_{\widetilde{\ZZ}}}(f^*\OO_\ZZ(-\partial\ZZ),\OO_{\widetilde{\ZZ}})), \,\,\,\text{by (\ref{pullbackhom})}\\\\
&=f_* \Hom_{\OO_{\widetilde{\ZZ}}}(f^*\OO_\ZZ(-\partial\ZZ),\omega_{\widetilde{\ZZ}})\\\\
&=\Hom_{\OO_{\ZZ}}(\OO_\ZZ(-\partial\ZZ),f_*\omega_{\widetilde{\ZZ}}), \,\,\,\text{by adjunction (cf. \cite[Chapter II, \S5]{H}})\\\\
&=\Hom_{\OO_{\ZZ}}(\OO_\ZZ(-\partial\ZZ),\omega_{\ZZ}), \,\,\,\text{by Proposition }\ref{ratlsing} \text{ and \cite[Theorem 5.10]{KM}}\\\\
&=\omega_\ZZ(\partial\ZZ)
\end{align*}

\end{proof}

\begin{lem}
The homogeneous line bundle $\mathcal{L}(\rho)|_{X_w}$ (cf. \S2) has a section with zero set exactly equal to $\partial X_w$.  
\end{lem}

\begin{proof}
Consider the Borel-Weil homomorphism (cf. \cite[\S8.1.21]{Kbook}) $\chi_w:L(\rho)^*\rightarrow H^0(X_w,\mathcal{L}(\rho))$ which is given by $\chi_w(f)(gB)=[g,f(ge_\rho)]$, where $e_\rho$ is the highest weight vector of the irreducible highest weight $G^\text{min}$-module $L(\rho)$ with highest weight $\rho$.  

Consider the section $\chi_w(e^*_{w\rho})$, where $e_{w\rho}$ is the weight vector of $L(\rho)$ with weight $w\rho$ and $e^*_{w\rho}\in L(\rho)^*$ is the linear form which takes the value $1$ on $e_{w\rho}$ and $0$ on any weight vector of $L(\rho)$ of weight different from $w\rho$.  Let $w'\leq w$ and $b\in B$.  We have $$\chi_w(e^*_{w\rho})(bw'B)=[bw',e^*_{w\rho}(bw'e_\rho)].$$  Now, for $w'<w$, we have $w'\rho>w\rho$ by \cite[Lemma 8.3.3]{Kbook}.  Thus, $e^*_{w\rho}(bw'e_\rho)=0$ for any $b\in B$ and $w'<w$.  For $w'=w$, we have $e^*_{w\rho}(bwe_\rho)\neq 0$ for any $b\in B$.  Hence, $\chi_w(e^*_{w\rho})$ has zero set precisely equal to $\partial X_w$.  
\end{proof}

\begin{lem}
\label{section}
There is an ample line bundle $\mathcal{L}$ on $(X_w)_j$ with a section with zero set exactly equal to $\partial((X_w)_j)$.
\end{lem}

\begin{proof}
By the previous lemma, the ample line bundle $\mathcal{L}(\rho)|_{X_w}$ has a section with zero set exactly equal to $\partial X_w$.  The $T$-equivariant line bundle $(e^{w\rho}\mathcal{L}(\rho))|_{X_w}$ gives rise to the line bundle $$(e^{w \rho}\mathcal{L}(\rho)|_{X_w})_j  :=E(T)_{\PP_j}\times^T\left((e^{w\rho}\mathcal{L}(\rho))|_{X_w}\right)$$ on $(X_w)_j$.  Then, the section $\theta$ defined by $[e,x]\mapsto[e,1_{w\rho}\otimes\chi_w(e^*_{w\rho})x]$ for $e\in E(T)_{\PP_j}$ and $x\in X_w$ has zero set exactly equal to $(\partial X_w)_j$, where 
$1_{w\rho}$ is a nonzero element of the line bundle $e^{w\rho}$ (cf. $\S$2).  

Now, let $\mathcal{H}$ be an ample line bundle on $\PP_j$ with a section $\sigma$ with zero set exactly $\partial\PP_j$ and consider the bundle $$\mathcal{L}:=(e^{w\rho}\mathcal{L}(\rho)|_{X_w})_j\otimes p^*(\mathcal{H}^{\bar{N}}),$$  
where $p: (X_w)_j \to \PP_j$ is the projection. Let $\hat{\sigma}$ represent the pullback of $\sigma$.  Then, the section $\theta\otimes (\hat{\sigma})^{\bar{N}}$ has zero set exactly equal to $\partial((X_w)_j)$.  Furthermore, by \cite[Proposition 1.45]{KM},  if $\bar{N}$ is large enough, $\mathcal{L}$ is ample.  
\end{proof}

\begin{lem}
\label{big}
Let $f:\widetilde{\ZZ}\to\ZZ$ be a proper birational map between normal, irreducible varieties and let $\pi:\ZZ\to\Gamma$ be a surjective 
proper morphism.  Let $\tilde{\mathcal{L}}$ be a $\pi$-big line bundle on $\ZZ$.  Then, the pullback line bundle $f^*\tilde{\mathcal{L}}$ is $\tilde{\pi}$-big, where $\tilde{\pi}=\pi\circ f:\widetilde{\ZZ}\to\Gamma$.  
\end{lem}

\begin{proof}
It suffices to show that $\text{rank }\tilde{\pi}_*(f^*\tilde{\mathcal{L}}^k)>c\cdot k^n$ for some $c>0$ and $k\gg 1$, where $n$ is the dimension of a general fiber of $\tilde{\pi}$.  

By \cite[Chapter 1, \S6.3,  Theorem 7]{S}, the dimension of a general fiber of $\pi$ and the dimension of a general fiber of $\tilde{\pi}$ are both equal to $\dim\ZZ-\dim\Gamma$.  

Now, $\tilde{\pi}_*(f^*\tilde{\mathcal{L}}^k) =\pi_*f_*(f^*\tilde{\mathcal{L}}^k) \simeq \pi_*(\tilde{\mathcal{L}}^k)$ by the projection formula, since $\ZZ$ is normal.  Hence, $\text{rank }\tilde{\pi}_*(f^*\tilde{\mathcal{L}}^k)>c\cdot k^n$ for some $c>0$ and $k\gg 1$, where $n$ is the dimension of a general fiber of $\pi$, since $\tilde{\mathcal{L}}$ is $\pi$-big.  This proves the lemma.

\end{proof}

\begin{thm}
\label{thmR^i}
Assume that $c^w_{u,v}(j)\neq 0$.  Then, for all $i>0$,  we have $R^i\tilde{\pi}_*\omega_{\widetilde{\ZZ}}(\partial\widetilde{\ZZ})=0$ and $R^i f_*\omega_{\widetilde{\ZZ}}(\partial\widetilde{\ZZ})=0$, where $\tilde{\pi}:\widetilde{\ZZ}\to\bar{\Gamma}$ is the projection onto the first factor
and $f: \widetilde{\ZZ} \to \ZZ$ is defined after Lemma \ref{msmooth}.
\end{thm}

\begin{proof}
Since $f$ is proper and birational by Proposition \ref{Z}, it is surjective.  By the proof of Corollary \ref{corNM1}, the projection $\pi:\ZZ\to\bar{\Gamma}$ is also surjective.  Thus, $\tilde{\pi}=\pi\circ f$ is also surjective.  Moreover, $\tilde{\pi}$ is proper since it is the restriction of the projection $\bar{\Gamma}\times (Z^{u,v}_w)_\PP\to \bar{\Gamma}$ to the closed subset $\widetilde{\ZZ}$, where $(Z^{u,v}_w)_\PP$ is projective.  Let $\mathcal{M}$ denote the line bundle $\mathcal{M}:=\OO_{\widetilde{\ZZ}}(\partial\widetilde{\ZZ})$ (observe that $\partial\widetilde{\ZZ}$ is a divisor in the non-singular scheme $\widetilde{\ZZ}$).  By Theorem \ref{KV},  it suffices to find a normal crossings divisor $D\subset \widetilde{\ZZ}$ such that the line bundle $\mathcal{M}^{\bar{N}}(-D)$ is $\tilde{\pi}$-nef, $f$-nef, $\tilde{\pi}$-big, and $f$-big and $D=\sum_{j=1}^r \,a_jD_j,$ with $0 < a_j<\bar{N}$.  

By  Lemma \ref{section} we may choose an ample line bundle $\mathcal{L}$ on $\tilde{\Delta}((X_w)_j)$ with a section with zero set exactly equal to $\tilde{\Delta}(\partial ((X_w)_j))$.  Let $\varphi:\tilde{\Delta}((Z_w)_j)\to\tilde{\Delta}((X_w)_j)$ denote the desingularization.  Since $\partial ((Z_w)_j)$ is a normal crossings divisor in $(Z_w)_j$, it follows that $\partial \widetilde{\ZZ}=\left(\bar{\Gamma}\times(Z^{u,v}_w)_\PP\right)\times_{(Z^2_w)_\PP}\tilde{\Delta}(\partial ((Z_w)_j))$ is a normal crossings divisor in $\widetilde{\ZZ}$.  Write $$\partial\widetilde{\ZZ}=D_1+\dots + D_\ell$$ where $D_i$ are the irreducible components.  Since $\tilde{\mu}^*\varphi^*\mathcal{L}$ has a section with zero set precisely equal to $\partial\widetilde{\ZZ}$, it follows that $$\tilde{\mu}^*\varphi^*\mathcal{L}=\OO_{\widetilde{\ZZ}}(b_1D_1+\dots+b_\ell D_\ell)$$ for some positive integers $b_1,\dots,b_\ell$.  

Let $D$ be the divisor $D:=a_1D_1+\dots+a_\ell D_\ell$ where $a_i:=\bar{N}-b_i$ for some integer $\bar{N}$ greater than all the $b_i$'s.  Since $\partial\widetilde{\ZZ}$ has normal crossings, so does $D$.  Then, $$\mathcal{M}^{\bar{N}}(-D)=\OO_{\widetilde{\ZZ}}(b_1D_1+\dots+b_\ell D_\ell)=\tilde{\mu}^*\varphi^*\mathcal{L}.$$  Since the fibers of $\tilde{\pi}$ are projective schemes and $\mathcal{L}$ is an ample line bundle on $\tilde{\Delta}((X_w)_j)$, the pull-back $\tilde{\mu}^*\varphi^*\mathcal{L}$ restricted to the fibers of $\tilde{\pi}$ is nef, since the pullback of any ample line bundle under a morphism between projective varieties is nef (cf. \cite[Theorem 1.26, \S1.9 and \S1.29]{D}).  Thus, $\mathcal{M}^{\bar{N}}(-D)$ is $\tilde{\pi}$-nef.  Since the fibers of $f$ are contained in the fibers of $\tilde{\pi}$, $\mathcal{M}^{\bar{N}}(-D)$ is also $f$-nef.  

Now, $\mathcal{M}^{\bar{N}}(-D)$ is $f$-big since $f$ is birational by Proposition \ref{Z}.  It remains to show $\tilde{\pi}$-bigness.  Clearly,  $\mu:\ZZ\to\tilde{\Delta}((X_w)_j)$ is a closed embedding restricted to any fiber of the morphism $\pi:\ZZ\to\bar{\Gamma}$.  Hence, the ample line bundle $\mathcal{L}$ on $\tilde{\Delta}((X_w)_j)$ pulls back to a $\pi$-big line bundle $\mu^*\mathcal{L}$ on $\ZZ$.  Now, $\tilde{\pi}$-bigness of $\mathcal{M}^{\bar{N}}(-D)=\tilde{\mu}^*\varphi^*\mathcal{L}=f^*\mu^*\mathcal{L}$ follows from Lemma \ref{big} and Proposition \ref{Z}.
\end{proof}

Theorems \ref{thmf_*} and \ref{thmR^i} together with the Grothendieck spectral sequence give

\begin{cor}
\label{corR^i}
For all $i>0$ we have $R^i\pi_*\omega_{\ZZ}(\partial\ZZ)=0$.
\end{cor}

\section{Proof of part b) of Theorem \ref{main}}

Recall the definitions of $N_\gamma$ and $M_\gamma$ from equations (\ref{eqNdef}) and (\ref{eqMdef}).

We also define $$M_\gamma^1:=\gamma((X^{u,v}_w)_{\PP})\cap\tilde{\Delta}((X_w)_{\partial\PP_j})$$ and $$M_\gamma^2:=\gamma((X^{u,v}_w)_\PP)\cap\tilde{\Delta}((\partial X_w)_j),$$ so that $$M_\gamma=M_\gamma^1 \cup M_\gamma^2.$$

\begin{lem}
\label{lemma2}
For general $\gamma\in\bar{\Gamma}$, the sheaf $\gamma_*\OO_{(X^u\times X^v)_\PP}\otimes_{\OO_{\Y_\PP}}\left({\hat{p}}^*(\OO_{\PP_j}(-\partial\PP_j))\otimes_{\OO_{Y_\PP}}\tilde{\Delta}_*((\xi_w)_\PP)\right)$ is supported on $N_\gamma$ and is equal to the sheaf $\OO_{N_\gamma}(-M_\gamma)$, where $\hat{p}:Y_\PP\to\PP$ is the projection.
\end{lem}

\begin{proof}
Since $\hat{p}:Y_\PP\to \PP$ is flat, we have a short exact sequence $$0\to {\hat{p}}^*(\OO_{\PP_j}(-\partial \PP_j))\to{\hat{p}}^*\OO_{\PP_j}\to{\hat{p}}^*\OO_{\partial\PP_j}\to 0.$$ By Lemma \ref{lemtor2} tensoring with $\tilde{\Delta}_*((\xi_w)_\PP)$ over $\OO_{Y_\PP}$ preserves exactness of the above sequence, so we have an exact sequence $$0\to {\hat{p}}^*(\OO_{\PP_j}(-\partial \PP_j))\otimes_{\OO_{Y_\PP}}\tilde{\Delta}_*((\xi_w)_\PP)\to {\hat{p}}^*\OO_{\PP_j}\otimes_{\OO_{Y_\PP}}\tilde{\Delta}_*((\xi_w)_\PP)\to {\hat{p}}^*\OO_{\partial\PP_j}\otimes_{\OO_{Y_\PP}}\tilde{\Delta}_*((\xi_w)_\PP)\to 0.$$  By Theorem \ref{maintor}, for general $\gamma\in\bar{\Gamma}$, tensoring the above sequence with $\gamma_*\OO_{(X^u\times X^v)_\PP}$ over $\OO_{\Y_\PP}$ preserves exactness, so we have, for general $\gamma\in\bar{\Gamma}$, an exact sequence 

\begin{equation*}
0 \to \gamma_*\OO_{(X^u\times X^v)_\PP}\otimes_{\OO_{\Y_\PP}}\left({\hat{p}}^*(\OO_{\PP_j}(-\partial \PP_j))\otimes_{\OO_{Y_\PP}}\tilde{\Delta}_*((\xi_w)_\PP)\right) \to \\
\end{equation*}
\begin{equation}
\label{exact1}
\gamma_*\OO_{(X^u\times X^v)_\PP}\otimes_{\OO_{\Y_\PP}} \left({\hat{p}}^*\OO_{\PP_j}\otimes_{\OO_{Y_\PP}}\tilde{\Delta}_*((\xi_w)_\PP)\right) \to \gamma_*\OO_{(X^u\times X^v)_\PP}\otimes_{\OO_{\Y_\PP}} \left({\hat{p}}^*\OO_{\partial \PP_j}\otimes_{\OO_{Y_\PP}}\tilde{\Delta}_*((\xi_w)_\PP)\right) \to  0.
\end{equation}

Next, consider the exact sequence 

\begin{equation}
\label{exact2}
0\to \OO_{N_\gamma}(-M_\gamma) \to \OO_{N_\gamma}(-M_\gamma^2) \to \OO_{M_\gamma^1}(-M_\gamma^1\cap M_\gamma^2) \to 0.
\end{equation}
Comparing (\ref{exact1}) with (\ref{exact2}) we see that it is sufficient to show that \begin{equation}
\label{show1}
\gamma_*\OO_{(X^u\times X^v)_\PP}\otimes_{\OO_{\Y_\PP}} \left({\hat{p}}^*\OO_{\PP_j}\otimes_{\OO_{Y_\PP}}\tilde{\Delta}_*((\xi_w)_\PP)\right)=\OO_{N_\gamma}(-M_\gamma^2)
\end{equation} 
and 
\begin{equation}
\label{show2}
\gamma_*\OO_{(X^u\times X^v)_\PP}\otimes_{\OO_{\Y_\PP}} \left({\hat{p}}^*\OO_{\partial \PP_j}\otimes_{\OO_{Y_\PP}}\tilde{\Delta}_*((\xi_w)_\PP)\right)=\OO_{M_\gamma^1}(-M_\gamma^1\cap M_\gamma^2).
\end{equation}

To prove (\ref{show1}) and (\ref{show2}) consider the short exact sequence

\begin{equation}
\label{ses}
0 \to \tilde{\Delta}_*((\xi_w)_\PP) \to \tilde{\Delta}_*(\OO_{(X_w)_\PP}) \to \tilde{\Delta}_*(\OO_{(\partial X_w)_\PP}) \to 0.
\end{equation}
Tensor the sequence (\ref{ses}) with ${\hat{p}}^*(\OO_{\PP_j})$ over $\OO_{Y_\PP}$, which preserves exactness by the proof of Lemma \ref{lemtor2}, and then tensor with $\gamma_*\OO_{(X^u\times X^v)_\PP}$ over $\OO_{\Y_\PP}$, which preserves exactness by Theorem \ref{maintor}.  For closed subschemes $X$ and $Y$ of a scheme $Z$, 
 $$\OO_X\otimes_{\OO_Z}\OO_Y=\OO_{X\cap Y}.$$  Thus,  $$\gamma_*\OO_{(X^u\times X^v)_\PP}\otimes_{\OO_{\Y_\PP}} \left({\hat{p}}^*(\OO_{\PP_j})\otimes_{\OO_{Y_\PP}} \tilde{\Delta}_*(\OO_{(X_w)_\PP})\right) =\OO_{\gamma(X^u\times X^v)_\PP\cap\tilde{\Delta}((X_w)_j)}=\OO_{N_\gamma}$$ and $$\gamma_*\OO_{(X^u\times X^v)_\PP}\otimes_{\OO_{\Y_\PP}} \left( {\hat{p}}^*(\OO_{\PP_j})\otimes_{\OO_{Y_\PP}}\tilde{\Delta}_*(\OO_{(\partial X_w)_\PP}) \right)=\OO_{\gamma(X^u\times X^v)_\PP\cap\tilde{\Delta}((\partial X_w)_j)}=\OO_{M^2_\gamma}.$$  Thus, (\ref{show1}) follows.

Similarly, (\ref{show2}) follows by tensoring the sequence (\ref{ses}) with ${\hat{p}}^*(\OO_{\partial \PP_j})$ over $\OO_{Y_\PP}$, which preserves exactness by the proof of Lemma \ref{lemtor2}, and then tensoring with $\gamma_*\OO_{(X^u\times X^v)_\PP}$ over $\OO_{\Y_\PP}$, which preserves exactness by Theorem \ref{maintor}.  This completes the proof.  
\end{proof}

By the previous lemma, Theorem \ref{main} part b) is equivalent to the following theorem:

\begin{thm}
\label{thm8.2}
For general $\gamma\in\bar{\Gamma}$ and any $u,v,w\in W,j\in[N]^r$ such that $c^w_{u,v}(j)\neq 0$ we have $$H^p(N_\gamma,\OO_{N_\gamma}(-M_\gamma)) =0 ,$$ for all $p\neq |j|+\ell(w)-\ell(u)-\ell(v)$.
\end{thm}

\begin{proof}

First, the theorem is equivalent to the statement that for general $\gamma\in\bar{\Gamma}$,
\begin{equation}
\label{cohvanish}
H^p(N_\gamma,\omega_{N_\gamma}(M_\gamma))=0,\text{ }\forall p>0.
\end{equation}

To see this, observe that:

\begin{align*}
H^p(N_\gamma,\omega_{N_\gamma}(M_\gamma)) &= H^p(N_\gamma,\Hom_{\OO_{N_\gamma}}(\OO_{N_\gamma}(-M_\gamma),\omega_{N_\gamma}))\\
&\overset{\varphi_1}\simeq \text{Ext}^p_{N_\gamma}(\OO_{N_\gamma}(-M_\gamma),\omega_{N_\gamma})\\
&\overset{\varphi_2}\simeq H^{n-p}(N_\gamma,\OO_{N_\gamma}(-M_\gamma))^*.
\end{align*}
where $n:=\dim N_\gamma=|j|+\ell(w)-\ell(u)-\ell(v)$, the isomorphism $\varphi_1$ follows by Corollary \ref{corNM2} and the local to global Ext spectral sequence \cite[Th\'eor\`eme 7.3.3, Chap. II]{Go}, and the isomorphism $\varphi_2$ follows by Corollary \ref{corNM1} and Serre duality \cite[Chap. III, Theorem 7.6]{H}.

We now prove (\ref{cohvanish}), which implies the theorem.
By \cite[Chapter I, \S6.3, Theorem 1.25]{S} and \cite[Chap. III, Exercise 10.9]{H}, there is a nonempty open subset $\mathcal{U}\subset\bar{\Gamma}$ such that $\pi:\pi^{-1}(\mathcal{U})\to\mathcal{U}$ is flat.  (Observe that, by the proof of Corollary \ref{corNM1}, $\pi$ is surjective.) We prove that  $\omega_\ZZ(\partial\ZZ)$ is flat over $\mathcal{U}$:  

To show this, let $A=\OO_\mathcal{U}$, $B=\OO_{\pi^{-1}(\mathcal{U})}$, and $M=\omega_\ZZ(\partial\ZZ)|_{\pi^{-1}(\mathcal{U})}$.  By taking stalks, we immediately reduce to showing that for an embedding of local rings $A\subset B$ such that $A$ is regular and $B$ is flat over $A$, we have that $M$ is flat over $A$.  Now, to prove this, let $\{x_1,\dots,x_d\}$ be a minimal set of generators of the maximal ideal of $A$.  Let $K_\bullet=K_\bullet(x_1,\dots,x_d)$ be the Koszul complex of the $x_i$'s over $A$.  Then, recall that a finitely generated $B$-module $N$ is flat over $A$ iff $K_\bullet\otimes_A N$ is exact except at the extreme right, i.e., $H^i(K_\bullet\otimes_A N)=0$ for $i<d$ \cite[Corollary 17.5 and Theorem 6.8]{E}.  Thus, by hypothesis, $K_\bullet\otimes_A B$ is exact except at the extreme right and hence  the $x_i$'s form a $B$-regular sequence  \cite[Theorem 17.6]{E}.  Now, since $\OO_\ZZ$ and $\OO_{\partial\ZZ}$ are CM, we have that $\OO_\ZZ(-\partial\ZZ)$ is a CM $\OO_\ZZ$-module.  Thus, by \cite[Proposition 11.33]{I}, we have that $M$ is a CM $B$-module of dimension equal to $\dim B$.  Therefore, by \cite[Exercise 11.36]{I}, the $x_i$'s form a regular sequence on the $B$-module $M$.  Hence, $(K_\bullet\otimes_A B)\otimes_B M\simeq K_\bullet\otimes_A M$ is exact except at the extreme right \cite[Corollary 17.5]{E}.  This proves that $M$ is flat over $A$, as desired.  

Thus, by Corollary \ref{corR^i} and the semicontinuity theorem \cite[Theorem 13.1]{Ke} to prove (\ref{cohvanish}), it is sufficient to show that for general $\gamma\in\bar{\Gamma}$, that 
\begin{equation}
\label{eqnnew}\omega_\ZZ(\partial \ZZ)|_{\pi^{-1}(\gamma)}\simeq\omega_{N_\gamma}(M_\gamma).
\end{equation}  
To prove this, observe that since $\mathcal{U}$ is smooth and $\ZZ$ and $\partial \ZZ$ are CM, and the assertion is local in $\mathcal{U}$, it suffices to observe (cf. \cite[Corollary 11.35]{I}) that for a nonzero function $\theta$ on $\mathcal{U}$, the sheaf $$\mathcal{S}/\theta\cdot\mathcal{S}=\Hom_{\OO_{\ZZ_\theta}}(\OO_\ZZ(-\partial \ZZ)/\theta\cdot\OO_\ZZ(-\partial \ZZ),\omega_{\ZZ_\theta}),$$ where $\ZZ_\theta$ denotes the zero scheme of $\theta$ in $\ZZ$ and the sheaf $\mathcal{S}:=\Hom_{\OO_\ZZ}(\OO_\ZZ(-\partial \ZZ),\omega_\ZZ)$.  Choosing $\theta$ to be in a local coordinate system and using induction and the above result, the desired conclusion (\ref{eqnnew}) is obtained.

\end{proof}

\vskip8ex
\noindent
Address: Department of Mathematics, University of North Carolina, Chapel Hill, NC 27599-3250, USA. 
\end{document}